\documentclass[12pt,a4paper,reqno,oneside]{amsart}

\usepackage{amsmath}
\usepackage{amsthm}
\usepackage{amssymb}
\usepackage{bbm}
\usepackage[sort,nocompress]{cite} 

\usepackage{amsfonts}
\usepackage{enumerate}

\usepackage{xcolor}
\usepackage{geometry}
\usepackage[colorlinks,citecolor=red]{hyperref}
\usepackage[notref,notcite,final]{showkeys}

\usepackage{graphicx}
\usepackage[capitalise]{cleveref}
\usepackage{courier}
\usepackage{listings} 
\lstdefinelanguage{Julia}%
  {morekeywords={abstract,break,catch,const,continue,do,else,elseif,%
      end,export,false,for,function,immutable,import,importall,if,in,%
      macro,module,otherwise,quote,return,switch,true,try,type,typealias,%
      using,while},%
   sensitive=true,%
   alsoother={$},%
   morecomment=[l]\#,%
   morecomment=[n]{\#=}{=\#},%
   morestring=[s]{"}{"},%
   morestring=[m]{'}{'},%
}[keywords,comments,strings]%

\newtheorem{lemma}{Lemma}[section]

\newtheorem{proposition}[lemma]{Proposition}
\newtheorem{theorem}[lemma]{Theorem}
\newtheorem{example}[lemma]{Example}

\newtheorem{setting}[lemma]{Setting}
\theoremstyle{definition}
\newtheorem{listing}[lemma]{Listing}

\renewcommand{\gets}{\curvearrowleft}
\providecommand{\N}{{\ensuremath{\mathbbm{N}}}}
\providecommand{\Z}{{\ensuremath{\mathbbm{Z}}}}
\providecommand{\R}{{\ensuremath{\mathbbm{R}}}}

\renewcommand{\P}{{\ensuremath{\mathbbm{P}}}}
\providecommand{\E}{{\ensuremath{\mathbbm{E}}}}

\providecommand{\1}{{\ensuremath{\mathbbm{1}}}}
\providecommand{\F}{{\ensuremath{\mathbbm{F}}}}

\newcommand{\exponentV}{{p_{\mathrm{v}}}}
\newcommand{\exponentZ}{{p_{\mathrm{z}}}}
\newcommand{\exponentX}{{p_{\mathrm{x}}}}
\newcommand{\exponentP}{{\mathfrak{p}}}
\newcommand{\exponentFirstNorm}{{q_1}}
\newcommand{\pr}{\mathrm{pr}}
\newcommand{\var}{\mathbbm{V}}

\newcommand{\threenorm}[1]{{\left\vert\kern-0.25ex\left\vert\kern-0.25ex\left\vert #1 
    \right\vert\kern-0.25ex\right\vert\kern-0.25ex\right\vert}}
\newcommand{\tnorm}[1]{{\left\vert\kern-0.25ex\left\vert\kern-0.25ex\left\vert #1     \right\vert\kern-0.25ex\right\vert\kern-0.25ex\right\vert}}
\newcommand{\xeqref}[1]{{\tiny\color{red}\eqref{#1}}} 
\newcommand{\supnorm}[1]{{\left\vert\kern-0.25ex\left\vert\kern-0.25ex\left\vert #1     \right\vert\kern-0.25ex\right\vert\kern-0.25ex\right\vert}}

\providecommand{\var}{{\ensuremath{\mathbbm{V}}}}

\author[T.A. Nguyen]{Tuan Anh Nguyen}
\address{Faculty of Mathematics, Bielefeld University, Germany}
\email{tnguyen@math.uni-bielefeld.de}
\title[MLP]{Multilevel Picard approximations overcome the curse of dimensionality when approximating semilinear heat equations with gradient-dependent nonlinearities in $L^p$-sense}
\begin{document}\begin{abstract}
We prove that 
multilevel Picard approximations are capable of approximating solutions of semilinear heat equations
in $L^{p}$-sense, ${p}\in [2,\infty)$,
in the case of gradient-dependent, Lipschitz-continuous nonlinearities, in the sense that the computational effort of the 
multilevel Picard approximations grow at most polynomially in both the dimension $d$ and the reciprocal $1/\epsilon$ of the 
prescribed  accuracy $\epsilon$. 
\end{abstract}
\maketitle

\allowdisplaybreaks

\section{Introduction}

Partial differential equations (PDEs) are important tools to analyze many real world phenomena, e.g., in financial engineering, economics, quantum mechanics, or statistical physics to name but a few. 
In most of the cases such high-dimensional nonlinear PDEs cannot be solved explicitly. It is one of
the most challenging problems in applied mathematics to approximately solve high-dimensional nonlinear PDEs.
In particular, it is very difficult to find approximation schemata for nonlinear PDEs for which one can
rigorously prove that they do overcome the so-called \emph{curse of dimensionality} in the sense that the computational complexity only grows polynomially in the space dimension $d$ of the PDE and the reciprocal
${1}/{\varepsilon}$ of the accuracy $\varepsilon$.

In recent years, 
there are two types of approximation methods which are quite
successful in the numerical approximation of solutions of high-dimensional nonlinear  PDEs:  neural network based 
approximation methods for PDEs, cf.,
\cite{al2022extensions,beck2020deep,beck2021deep,beck2019machine,
beck2020overview,berner2020numerically,castro2022deep,CHW2022,
ew2017deep,ew2018deep,weinan2021algorithms,frey2022convergence,frey2022deep,GPW2022,
gnoatto2022deep,gonon2023random,gonon2021deep,gonon2023deep,grohs2023proof,
han2018solving,han2020convergence, 
han2019solving,hure2020deep,hutzenthaler2020proof,ito2021neural,
jacquier2023deep,jacquier2023random,JSW2021,
lu2021deepxde,nguwi2022deep,nguwi2022numerical,nguwi2023deep,
raissi2019physics,reisinger2020rectified,sirignano2018dgm,zhang2020learning,han2018solving,NSW2024}   
and multilevel Monte-Carlo based approximation methods for PDEs, cf.,
\cite{beck2020overcomingElliptic,beck2020overcoming,becker2020numerical,
hutzenthaler2019multilevel,hutzenthaler2021multilevel,
giles2019generalised,HJK2022,HJK+2020,
HJKN2020,
hutzenthaler2020overcoming,HK2020,hutzenthaler2022multilevel,
HN2022a,NW2022,NNW2023,NW2023,HJKP2021}.

For multilevel Monte-Carlo based algorithms it is often possible to provide a complete convergence and complexity analysis.
It has been proven that under some suitable assumptions, e.g., Lipschitz continuity
on the linear part, the nonlinear part, and the initial (or terminal) condition function of the PDE under consideration, 
the multilevel Picard (MLP) approximation algorithms can overcome the curse of
dimensionality in the sense that the number of computational operations of
the proposed Monte-Carlo based approximation method grows at most polynomially in both the reciprocal
${1}/{\varepsilon}$ of the prescribed approximation accuracy $\varepsilon\in (0,1)$ and the PDE dimension $d\in \{1,2,\ldots\}$.
More precisely, \cite{hutzenthaler2021multilevel} considers smooth semilinear parabolic heat equations. 
Later, \cite{HJK+2020} extends \cite{hutzenthaler2021multilevel} to a more general setting, namely, semilinear heat equations which are not necessary smooth. \cite{beck2020overcoming} considers semilinear heat equation with more general nonlinearities, namely locally Lipschitz nonlinearities.
 \cite{HK2020,HJK2022} consider semilinear heat equations with gradient-dependent Lipschitz nonlinearities and \cite{NNW2023,NW2023} extend them to semilinear PDEs with general drift and diffusion coefficients.
\cite{hutzenthaler2020overcoming} studies Black-Scholes-types semilinear PDEs.
\cite{HJKN2020,HN2022b} consider semilinear parabolic PDEs with nonconstant drift and diffusion coefficients.
\cite{HN2022a} considers a slightly more general setting than \cite{HJKN2020}, namely semilinear PDEs with locally monotone coefficient functions. 
\cite{NW2022} studies semilinear partial integro-differential equations.
\cite{hutzenthaler2022multilevel} considers McKean-Vlasov stochastic differential equations (SDEs) with constant diffusion coefficients. \cite{beck2020overcomingElliptic} studies a special type of elliptic equations.
Almost all the works listed above prove $L^2$-error estimates except
 \cite{HJKP2021,HN2022b}, which draw their attention to $L^\mathfrak{p}$-error estimates, $\mathfrak{p}\in [2,\infty)$. Note that both  \cite{HJKP2021,HN2022b} consider PDEs with \emph{gradient-independent} nonlinearities.

The main novelty of our paper is the following:
 We extend the $L^2$-complexity analysis in \cite{HJK2022} to an $L^\mathfrak{p}$-complexity analysis, $\mathfrak{p}\in [2,\infty)$. More precisely, in our  main result, \cref{a41} below, we introduce an MLP algorithm in \eqref{a03} and prove that it overcomes the curse of dimensionality when approximating semilinear heat equations with \emph{gradient-dependent} nonlinear parts in $L^\mathfrak{p}$-sense, $\mathfrak{p}\in [2,\infty)$. Compared to the $L^2$-case the main difficulty in the $L^\mathfrak{p}$-case is that the growth
of the number of samples used to approximate expectations via Monte Carlo averages must be
more carefully chosen (cf.\ the definition of $(M_n)_{n\in \N}$ in \eqref{c12}). Moreover, also compared to \cite{HJK2022}, our paper introduces a much shorter complexity analysis.

\subsection{Notations}

Throughout this paper we use the following notations. 
Let 
$\R$ denote the set of all real numbers. Let
$\Z, \N_0, \N $ denote the sets which satisfy that $\Z=\{\ldots,-2,-1,0,1,2,\ldots\}$, $\N=\{1,2,\ldots\}$, $\N_0=\N\cup\{0\}$. Let $\nabla$ denote the gradient  and $\Delta$ denote the Laplacian. 
For every probability space $(\Omega,\mathcal{F},\P)$, every random variable
$X\colon \Omega\to \R$, and every $s\in [1,\infty)$ let 
$\lVert X\rVert_s \in [0,\infty]$ satisfy that 
$\lVert X\rVert_s=(\E[\lvert X\rvert^s])^{\frac{1}{s}}$.
Denote by $\mathrm{B}(\cdot,\cdot)$ the beta function, i.e.,
$\mathrm{B}(z_1,z_2)=\int_{0}^{1}t^{z_1-1}(1-t)^{z_2-1}\,dt$
for complex numbers $z_1,z_2$ with $\min \{\Re(z_1),\Re(z_2)\}>0$. For a set $A$ denote by $\1_A$ the indicator function of $A$.

 When applying a result we often use a phrase like ``Lemma 3.8 with $d\gets (d-1)$''
that should be read as ``Lemma 3.8 applied with $d$ (in the notation of Lemma 3.8) replaced
by $(d-1 )$ (in the current notation)''.

\begin{theorem}\label{a41}
Let  $\Theta=\bigcup_{n\in \N}\Z^n$,
$T, \mathbf{k}\in (0,\infty)$, $\exponentP\in [2,\infty)$,
$\beta,c\in [1,\infty)$. For every $d\in \N$
let 
$(L^d_i)_{i\in [0,d]\cap \Z}\in \R^{d+1}$ satisfy that $
\sum_{i=0}^{d}L_i^d\leq c$.
For every $d\in \N$ let $\lVert\cdot \rVert\colon \R^d\to [0,\infty)$ be the standard norm on $\R^d$. 
For every $d\in \N$ let 
 $\Lambda^d=(\Lambda^d_{\nu})_{\nu\in [0,d]\cap\Z}\colon [0,T]\to \R^{1+d}$ satisfy for all $t\in [0,T]$ that $\Lambda^d(t)=(1,\sqrt{t},\ldots,\sqrt{t})$. 
For every $d\in \N$ let $\pr^d=(\pr^d_\nu)_{\nu\in [0,d]\cap\Z}\colon \R^{d+1}\to\R$ satisfy
for all 
$w=(w_\nu)_{\nu\in [0,d]\cap\Z}$,
$i\in [0,d]\cap\Z$ that
$\pr_i^d(w)=w_i$.
For every $d\in \N$ let $f_d\in C( [0,T)\times\R^d\times \R^{d+1},\R)$, $g_d\in C(\R^d,\R)$.
To shorten the notation we write 
for all $d\in \N$,
$t\in [0,T)$, $x\in \R^d$,
$w\colon[0,T)\times\R^d\to \R^{d+1} $ that
\begin{align}
(F_d(w))(t,x)=f_d(t,x,w(t,x)).
\end{align}
For every $d\in \N$ let $v_d\in C^{1,2}( [0,T]\times \R^d, \R)$ be an at most polynomially growing function. Assume for every $d\in \N$, $t\in (0,T)$, $x\in \R^d$ that
\begin{align}
\frac{\partial v_d}{\partial t}(t,x)+(\Delta_x v_d )(t,x)+
f_d(t,x,v_d(t,x), (\nabla_x v_d)(t,x))=0,\quad v_d(T,x)=g_d(x).\label{c10}
\end{align}
Let $\varrho\colon \{(\tau,\sigma)\in [0,T)^2\colon\tau<\sigma \}\to\R$ satisfy for all $t\in [0,T)$, $s\in (t,T)$ that
\begin{align}
\varrho(t,s)=\frac{1}{\mathrm{B}(\tfrac{1}{2},\tfrac{1}{2})}\frac{1}{\sqrt{(T-s)(s-t)}}.\label{a98c}
\end{align}
Let $ (\Omega,\mathcal{F},\P)$ be a probability space.
Let $\mathfrak{r}^\theta\colon \Omega\to (0,1) $, $\theta\in \Theta$, be independent and identically distributed\ random variables and satisfy for all
$b\in (0,1)$
that 
\begin{align}\label{a98b}
\P(\mathfrak{r}^0\leq b)=
\frac{1}{\mathrm{B}(\tfrac{1}{2},\tfrac{1}{2})}
\int_{0}^b\frac{dr}{\sqrt{r(1-r)}}.
\end{align}
For every $d\in \N$ let $W^{d,\theta}\colon[0,T]\times \Omega\to \R^d$, $\theta\in \Theta$, be independent standard Brownian motions.
Assume that
$
(W^{d,\theta})_{d\in \N,\theta\in \Theta}
$ and
$(\mathfrak{r}^\theta)_{\theta\in \Theta}$  are independent. Assume
for all $d\in \R^d$,
$t\in [0,T]$,
 $x,y\in \R^d$, $w_1,w_2\in \R^{d+1}$ that
 \begin{align}\label{a04g}
\max\{\lvert g_d(x)\rvert ,\lvert Tf_d(t,x,0)\rvert\}\leq
cd^\beta (d^c+\lVert x\rVert)^\beta,
\end{align}
\begin{align}
&\lvert
f_d(t,x,w_1)-f_d(t,y,w_2)\rvert\nonumber\\
&\leq \sum_{\nu=0}^{d}\left[
L_\nu^d\Lambda^d_\nu(T)
\lvert\pr^d_\nu(w_1-w_2) \rvert\right]
+\frac{1}{T}\frac{cd^\beta(d^c+\lVert x\rVert)^\beta
+cd^\beta(d^c+\lVert y\rVert)^\beta
}{2}\frac{\lVert x-y\rVert}{\sqrt{T}}
,\label{a05a}
\end{align}
\begin{align}
\lvert 
g(x)-g(y)\rvert\leq 
\frac{cd^\beta(d^c+\lVert x\rVert)^\beta
+cd^\beta(d^c+\lVert y\rVert)^\beta
}{2}\frac{\lVert x-y\rVert}{\sqrt{T}}\label{a09}
\end{align}
For every $d\in\N$, $\theta\in \Theta$, $t\in [0,T]$, 
$s\in [t,T]$,
$x\in \R^d$ let \begin{align}
\mathcal{X}^{d,\theta,t,x}_s=x+W^{d,\theta}_s-W^{d,\theta}_t
.\label{c01}
\end{align}
For every $d\in\N$, $\theta\in \Theta$, $t\in [0,T)$, 
$s\in (t,T]$,
$x\in \R^d$ let 
\begin{align}\label{c02}
\mathcal{Z}^{d,\theta,t,x}_s=\left(1,\frac{W^{d,\theta}_s-W^{d,\theta}_t}{s-t}\right)
.
\end{align}
Let $
U^{d,\theta}_{n,m}\colon [0,T)\times \R^d\times\Omega\to \R^{d+1}$,
$d\in \N$, $n,m\in \Z$, $\theta\in \Theta$, satisfy for all $d,n,m\in \N$, 
$\theta\in \Theta$,
$t\in [0,T)$, $x\in\R^d$ that
$U_{-1,m}^{d,\theta}(t,x)=U^{d,\theta}_{0,m}(t,x)=0$ and
\begin{align} \begin{split} 
&
U^{d,\theta}_{n,m}(t,x)= (g_d(x),0)+\sum_{i=1}^{m^n}
\frac{g_d(\mathcal{X}^{d,(\theta,0,-i),t,x}_T )-g_d(x) }{m^n}\mathcal{Z}^{d,(\theta,0,-i),t,x}_{T}\\
& +\sum_{\ell=0}^{n-1}\sum_{i=1}^{m^{n-\ell}} \frac{\left(F_d(U^{d,(\theta,\ell,i)}_{\ell,m})-
\1_\N(\ell)
F_d(U^{d,(\theta,\ell,-i)}_{\ell-1,m})\right)\!\left(t+(T-t) \mathfrak{r}^{(\theta,\ell,i)},
\mathcal{X}^{d,(\theta,\ell,i),t,x}_{t+(T-t) \mathfrak{r}^{(\theta,\ell,i)}}\right)
\mathcal{Z}^{d,(\theta,\ell,i),t,x}_{t+(T-t) \mathfrak{r}^{(\theta,\ell,i)}}}{m^{n-\ell}\varrho(t,t+(T-t)\mathfrak{r}^{(\theta,\ell,i)})}.\end{split}\label{a03}
\end{align}
Let $M\colon \N\to \N$ satisfy for all $n\in\N $ that
\begin{align}
M_n=\max \{k\in \N\colon k\leq\exp (\lvert\ln(n)\rvert^{1/2})\}.\label{c12}
\end{align} For every $d,m\in \N$, $n\in \N_0$ let $\mathrm{RV}_{d,n,m}$ be the number of realizations of scalar random variables which are used to compute one realization of $U^{d,0}_{n,m}(t,x,\omega)$.
Then the following items are true.
\begin{enumerate}[(i)]
\item \label{a16l} For all $d,n,m\in \N$, $t\in [0,T)$, $x\in \R^d$ we have that $U^{d,\theta}_{n,m}(t,x)$ is measurable.\item 
There exist
$\eta\in (0,\infty)$,
$(n(d,\epsilon))_{d\in \N,\epsilon\in (0,1)}\subseteq \N$, $(C_\delta)_{\delta\in (0,1)}\subseteq (0,\infty)$ such that for all $d\in \N$, $\epsilon, \delta\in (0,1)$ we have that
\begin{align}\begin{split} 
&
\sup_{{\nu\in [0,d]\cap\Z}}
\sup_{x\in [-\mathbf{k},\mathbf{k}]^d}
\Lambda^d_\nu(T)\left\lVert\pr^d_{\nu}\!\left(U^{d,0}_{n(d,\epsilon),M_{n(d,\epsilon)}}(0,x)-{u}_d(0,x)\right)\right\rVert_\exponentP
\leq \epsilon\end{split}
\end{align}
and
\begin{align}
&
\sum_{n=1}^{n(d,\epsilon)} \mathrm{RV}_{d,n,M_n}
\leq \eta d^\eta C_\delta \epsilon^{-(2+\delta)}
\end{align}
where $u_d:=(v_d,\nabla v_d)$.
\end{enumerate}
\end{theorem}The proof of \cref{a41} is presented in \cref{c11}. Let us comment
on the mathematical objects in \cref{a41}. Our goal here is to approximately solve the family of semilinear heat equations in \eqref{c10}. The functions $f_d$ are the nonlinear parts of the PDEs in \eqref{c10} which depend also on the gradients of the solutions. The functions $g_d$ are the terminal conditions at time $T$ of the PDEs in \eqref{c10}. Conditions \eqref{a04g}--\eqref{a09} are some regularity assumptions on $f_d$ and $g_d$. The filtered probability space $(\Omega,\mathcal{F}, \P, (\F_{t})_{t\in [0,T]} )$ in \cref{a41} above is the probability space on
which we introduce the stochastic MLP approximations which we employ to approximate the
solutions $v_d$ and their gradients $\nabla_x v_d$ of the PDEs in \eqref{c10}. The set $\Theta$ in \cref{a41} is used as an index set to introduce sufficiently many independent random variables.
The functions $\mathfrak{r}^\theta$ are independent random variables which are $\mathrm{Beta}(0.5,0.5)$-distributed on $(0,1)$ (see \eqref{a98b}). 
The functions $W^{d,\theta}$ describe
 independent standard Brownian motions which we use as random input
sources for the MLP approximations.
The functions $U^{d,\theta}_{n,m}$ in \eqref{a03} describe the MLP approximations which we employ to approximately compute the solutions $v_d$ and their gradients $\nabla_x v_d$ to the PDEs in \eqref{c10}.
The MLP approximations we use here are slightly different from that in \cite{HJK2022}. In fact, here we use a Beta distribution (see \eqref{a98b}) and the sequence $(M_n)_{n\in \N}$ (see \eqref{c12}). 
\cref{a41} establishes that the solutions $v_d$ of the PDEs in \eqref{c10}
can be approximated by the MLP approximations $U^{d,\theta}_{n,M_n}$
 in \eqref{a03} with the number of involved scalar random variables growing at most polynomially in the reciprocal
$1/\epsilon$ of the prescribed approximation accuracy $\epsilon\in (0,1)$ and at most polynomially in the PDE
dimension $d\in \N$. In other words, \cref{a41} states that MLP approximations overcome the curse of dimensionality when approximating the semilinear heat equations in~\eqref{c10}.

The paper is organized as follows. In \cref{s01a} we establish an
$L^\mathfrak{p}$-error bound for MLP approximations in an abstract setting and in \cref{c11} we prove the main result, \cref{a41}. In \cref{s12} we do two numerical experiments to illustrate our MLP algorithms.

\section{$L^\mathfrak{p}$-Error bound for the abstract MLP approximations}\label{s01a}
In this section we modify \cite[Section~4]{neufeld2023multilevel} to get $L^\mathfrak{p}$-estimates, $\mathfrak{p}\in [2,\infty)$.
First of all, in \cref{a41b} below we introduce an abstract setting for MLP approximations in the gradient-dependent case. This setting can be used for  more general frameworks in future research.
\begin{setting}\label{a41b}
Let $d\in \N$, $\Theta=\cup_{n\in \N}\Z^n$,
$T\in (0,\infty)$, $\exponentP\in [2,\infty)$,
$\exponentV,\exponentZ, \exponentX\in (1,\infty)$,
$c\in [1,\infty)$, $(L_i)_{i\in [0,d]\cap \Z}\in \R^{d+1}$ satisfy that $
\sum_{i=0}^{d}L_i\leq c$. 
Let $\lVert\cdot \rVert\colon \R^d\to [0,\infty)$ be a norm on $\R^d$. 
Let 
 $\Lambda=(\Lambda_{\nu})_{\nu\in [0,d]\cap\Z}\colon [0,T]\to \R^{1+d}$ satisfy for all $t\in [0,T]$ that $\Lambda(t)=(1,\sqrt{t},\ldots,\sqrt{t})$. 
Let $\pr=(\pr_\nu)_{\nu\in [0,d]\cap\Z}\colon \R^{d+1}\to\R$ satisfy
for all 
$w=(w_\nu)_{\nu\in [0,d]\cap\Z}$,
$i\in [0,d]\cap\Z$ that
$\pr_i(w)=w_i$.
Let $f\in C( [0,T)\times\R^d\times \R^{d+1},\R)$, $g\in C(\R^d,\R)$, 
$V\in C([0,T)\times \R^d, [0,\infty) )$ satisfy that
$\max \{ c,48e^{86c^6T^3}\}\leq V$. 
To shorten the notation we write 
for all 
$t\in [0,T)$, $x\in \R^d$,
$w\colon[0,T)\times\R^d\to \R^{d+1} $ that
\begin{align}
(F(w))(t,x)=f(t,x,w(t,x)).
\end{align}
Let $\varrho\colon \{(\tau,\sigma)\in [0,T)^2\colon\tau<\sigma \}\to\R$ satisfy for all $t\in [0,T)$, $s\in (t,T)$ that
\begin{align}
\varrho(t,s)=\frac{1}{\mathrm{B}(\tfrac{1}{2},\tfrac{1}{2})}\frac{1}{\sqrt{(T-s)(s-t)}}.\label{a98}
\end{align}
Let $ (\Omega,\mathcal{F},\P)$ be a probability space.
Let $\mathfrak{r}^\theta\colon \Omega\to (0,1) $, $\theta\in \Theta$, be independent and identically distributed\ random variables and satisfy for all
$b\in (0,1)$
that 
\begin{align}
\P(\mathfrak{r}^0\leq b)=
\frac{1}{\mathrm{B}(\tfrac{1}{2},\tfrac{1}{2})}
\int_{0}^b\frac{dr}{\sqrt{r(1-r)}}.\label{c14}
\end{align}
Let $\mathcal{X}^\theta=
(\mathcal{X}^{\theta,s,x}_t)_{s\in [0,T],t\in[s,T],x\in\R^d}\colon \{(\sigma,\tau)\in [0,T]^2\colon \sigma\leq \tau\}\times \R^d\times \Omega\to\R^d $, $\theta\in \Theta$, be measurable.
Let 
$\mathcal{Z}^\theta=
(\mathcal{Z}^{\theta,s,x}_t)_{s\in [0,T),t\in(s,T],x\in\R^d}\colon \{(\sigma,\tau)\in [0,T]^2\colon \sigma< \tau\}\times \R^d\times \Omega\to\R^{d+1} $, $\theta\in \Theta$, be measurable.
Assume that
$(\mathcal{X}^\theta, \mathcal{Z}^\theta)
$, $\theta\in \Theta$, are independent and identically distributed.
Assume that
$
(\mathcal{X}^\theta, \mathcal{Z}^\theta)_{\theta\in \Theta}
$ and
$(\mathfrak{r}^\theta)_{\theta\in \Theta}$  are independent. Assume
for all 
$i\in [0,d]\cap\Z$,
$s\in [0,T)$,
$t\in [s,T)$, $r\in (t,T]$,
 $x,y\in \R^d$, $w_1,w_2\in \R^{d+1}$
 that
\begin{align}\label{a04e}
\lvert g(x)\rvert\leq V(T,x),\quad \lvert Tf(t,x,0)\rvert\leq V(t,x),
\end{align}
\begin{align}
\lvert
f(t,x,w_1)-f(t,y,w_2)\rvert\leq \sum_{\nu=0}^{d}\left[
L_\nu\Lambda_\nu(T)
\lvert\pr_\nu(w_1-w_2) \rvert\right]
+\frac{1}{T}\frac{V(t,x)+V(t,y)}{2}\frac{\lVert x-y\rVert}{\sqrt{T}}
,\label{a01e}
\end{align}
\begin{align}\label{a02e}
\left\lVert
V(r,\mathcal{X}^{0,t,x}_r)
\right\rVert_{\exponentV}
\leq V(t,x),\quad 
\left\lVert\left\lVert\mathcal{X}^{0,s,x}_{t} -x\right\rVert\right\rVert_\exponentX\leq V(s,x)\sqrt{t-s},
\quad \left\lVert
\pr_i(\mathcal{Z}^{0,t,x}_r)\right\rVert_{\exponentZ}
\leq \frac{c}{\Lambda_{i}(r-t)},
\end{align}
\begin{align}V(T,x)\leq V(t,x),\quad 
\lvert g(x)-g(y)\rvert\leq \frac{V(T,x)+V(T,y)}{2}
\frac{\lVert x-y\rVert}{\sqrt{T}},\label{a19d}
\end{align}
\begin{align}
\P\!\left(\pr_0(\mathcal{Z}_{t}^{0,s,x})=1\right)=1,\quad \E\!\left[\mathcal{Z}_{t}^{0,s,x}\right]=(1,0,\ldots,0).\label{b03}
\end{align} 
Let $
U^{\theta}_{n,m}\colon [0,T)\times \R^d\times \Omega\to \R^{d+1}$, $n,m\in \Z$, $\theta\in \Theta$, satisfy for all $n,m\in \N$, 
$\theta\in \Theta$,
$t\in [0,T)$, $x\in\R^d$ that
$U_{-1,m}^\theta(t,x)=U^\theta_{0,m}(t,x)=0$ and
\begin{align} \begin{split} 
&
U^{\theta}_{n,m}(t,x)= (g(x),0)+\sum_{i=1}^{m^n}
\frac{g(\mathcal{X}^{(\theta,0,-i),t,x}_T )-g(x) }{m^n}\mathcal{Z}^{(\theta,0,-i),t,x}_{T}\\
& +\sum_{\ell=0}^{n-1}\sum_{i=1}^{m^{n-\ell}} \frac{\left(F(U^{(\theta,\ell,i)}_{\ell,m})-
\1_\N(\ell)
F(U^{(\theta,\ell,-i)}_{\ell-1,m})\right)\!\left(t+(T-t) \mathfrak{r}^{(\theta,\ell,i)},
\mathcal{X}^{(\theta,\ell,i),t,x}_{t+(T-t) \mathfrak{r}^{(\theta,\ell,i)}}\right)
\mathcal{Z}^{(\theta,\ell,i),t,x}_{t+(T-t) \mathfrak{r}^{(\theta,\ell,i)}}}{m^{n-\ell}\varrho(t,t+(T-t)\mathfrak{r}^{(\theta,\ell,i)})}.\end{split}\label{a04d}
\end{align}
\end{setting}

In \cref{a02} below we first study independence and distributional properties of MLP approximations. This has been done in the $L^2$-case. However, we state the result here for convenience of the reader.
\begin{lemma}[Independence and distributional properties]\label{a02}
Assume \cref{a41b}. Then the following items hold.
\begin{enumerate}[(i)]
\item \label{i01}We have for all $n\in \N_0$, $m\in \N$, $\theta\in \Theta$ that $U^\theta_{n,m} $ is measurable,
\item\label{i02} We have for all 
$n\in \N_0$, $m\in \N$, $\theta\in \Theta$ that
\begin{align}
&
\sigma( \{U^\theta_{n,m}(t,x)\colon  t\in [0,T), x\in \R^d \})\nonumber\\
&
\subseteq \sigma(
\{
\mathfrak{r}^{(\theta,\nu)},
\mathcal{X}^{(\theta,\nu),s,x}_{t}, 
\mathcal{Z}^{(\theta,\nu),s,x}_{t}
\colon 
\nu\in \Theta, s\in [0,T), t\in (s,T] , x\in \R^d\})
\end{align}
\item\label{i03} We have for all ${\theta\in \Theta}$, $m\in \N$ that
$(U^{\theta,\ell, i}_{\ell,m})_{t\in [0,T),x\in \R^d}$,
$(U^{\theta,\ell, -i}_{\ell-1,m})_{t\in [0,T),x\in \R^d}$,
$
((\mathcal{X}^{(\theta,\ell,i),s,x}_t)_{s\in [0,T],t\in[s,T],x\in\R^d}, 
(\mathcal{Z}^{(\theta,\ell,i),s,x}_t)_{s\in [0,T),t\in(s,T],x\in\R^d}   )
$,
$
\mathfrak{r}^{(\theta,\ell,i)}$, $i\in \N$, $\ell\in \N_0$, are independent,
\item \label{i04}We have for all $n\in \N_0$, $m\in \N$ that 
$(U^\theta_{n,m}(t,x ))_{t\in [0,T),x\in\R^d}$, $\theta\in \Theta$, are identically distributed.
\item\label{i05} We have for all $\theta\in \Theta$, $\ell\in \N_0$, $m\in \N$, $t\in [0,T)$, $x\in \R^d$ that 
\begin{align}
\frac{\left(F(U^{(\theta,\ell,i)}_{\ell,m})-
\1_\N(\ell)
F(U^{(\theta,\ell,-i)}_{\ell-1,m})\right)\!\left(t+(T-t) \mathfrak{r}^{(\theta,\ell,i)},
\mathcal{X}^{(\theta,\ell,i),t,x}_{t+(T-t) \mathfrak{r}^{(\theta,\ell,i)}}\right)
\mathcal{Z}^{(\theta,\ell,i),t,x}_{t+(T-t) \mathfrak{r}^{(\theta,\ell,i)}}}{\varrho(t,t+(T-t)\mathfrak{r}^{(\theta,\ell,i)})}, \quad i\in \N,
\end{align}
are independent and identically distributed and have the same distribution as 
\begin{align}
\frac{\left(F(U^{0}_{\ell,m})-
\1_\N(\ell)
F(U^{1}_{\ell-1,m})\right)\!\left(t+(T-t) \mathfrak{r}^{0},
\mathcal{X}^{0,t,x}_{t+(T-t) \mathfrak{r}^{0}}\right)
\mathcal{Z}^{0,t,x}_{t+(T-t) \mathfrak{r}^{0}}}{\varrho(t,t+(T-t)\mathfrak{r}^{0})}, \quad i\in \N.
\end{align}
\end{enumerate}
\end{lemma}

\begin{proof}[Proof of \cref{a02}]
The assumptions on measurability and distributions, basic properties of measurable functions, and induction prove 
\eqref{i01} and \eqref{i02}. In addition, \eqref{i02} and the assumptions on independence prove \eqref{i03}. Furthermore, \eqref{i03}, the fact that $\forall\theta\in \Theta, m\in \N\colon U^\theta_{0,m}=0$, \eqref{a04d}, the disintegration theorem, the assumptions on distributions, and induction establish \eqref{i04} and \eqref{i05}.
\end{proof}
\cref{a20} below can be considered as an $L^\mathfrak{p}$-version of \cite[Proposition~4.3]{neufeld2023multilevel}.
In \cref{a20} below we establish a recursive error bound for MLP approximations, see \eqref{a05f}. From this recursive error bound we get an error estimate for MLP approximations, see~\eqref{b40b}. The main novelty in the $L^\mathfrak{p}$-case is the Marcinkiewicz-Zygmund inequality (see \cite[Theorem~2.1]{Rio2009} and see \eqref{c60} below for an application of this inequality).
\begin{proposition}[Error analysis by semi-norms]\label{a20}
Assume \cref{a41b}. Let $\exponentFirstNorm\in [3,\infty)$.
Assume that
$\frac{1}{\exponentV}+\frac{1}{\exponentX}+\frac{1}{\exponentZ}\leq \frac{1}{\exponentP}$. For every random field $H\colon [0,T)\times \R^d\times \Omega\to \R^{d+1}$ let $\threenorm{H}_s$, $s\in [0,T)$, satisfy for all  $s\in [0,T)$ that
\begin{align}
\threenorm{H}_{s}=\max_{\nu\in [0,d]\cap\Z}\sup_{r\in [s,T),x\in \R^d}\frac{\Lambda_\nu(T-r)\left\lVert\pr_{\nu} (H(r,x))\right\rVert_\exponentP}{V^\exponentFirstNorm(r,x)}. \label{a25}
\end{align}
Then the following items hold.
\begin{enumerate}[(i)]
\item\label{a16} For all $n,m\in \N$, $t\in [0,T)$, $x\in \R^d$ we have that $U^{\theta}_{n,m}(t,x)$ is measurable.
\item \label{a17c}
There exists a unique measurable function $\mathfrak{u}\colon [0,T)\times \R^d\to \R^{d+1}$ such that for all $t\in [0,T)$, $x\in \R^d$
we have that
\begin{align}\label{b01e}
\max_{\nu\in [0,d]\cap\Z}\sup_{\tau\in [0,T), \xi\in \R^d}
\left[\Lambda_\nu(T-\tau)\frac{\lvert\pr_\nu(\mathfrak{u}(\tau,\xi))\rvert}{V(\tau ,\xi)}\right]<\infty,
\end{align}
\begin{align}\max_{\nu\in [0,d]\cap\Z}\left[
\E\!\left [\left\lvert g(\mathcal{X}^{0,t,x}_{T} )\pr_\nu(\mathcal{Z}^{0,t,x}_{T})\right\rvert \right] + \int_{t}^{T}
\E \!\left[\left\lvert
f(r,\mathcal{X}^{0,t,x}_{r},\mathfrak{u}(r,\mathcal{X}^{0,t,x}_{r}))\pr_\nu(\mathcal{Z}^{0,t,x}_{r})\right\rvert\right]dr\right]<\infty,
\end{align}
and
\begin{align}
\mathfrak{u}(t,x)=\E\!\left [g(\mathcal{X}^{0,t,x}_{T} )\mathcal{Z}^{0,t,x}_{T} \right] + \int_{t}^{T}
\E \!\left[
f(r,\mathcal{X}^{0,t,x}_{r},\mathfrak{u}(r,\mathcal{X}^{0,t,x}_{r}))\mathcal{Z}^{0,t,x}_{r}\right]dr.\label{a05e}
\end{align}
\item \label{a42}
For all $n,m\in \N$, $t\in [0,T)$ we have that
\begin{align}\label{a05f}
\threenorm{U^0_{n,m}-\mathfrak{u}}_t\leq 
\frac{16\sqrt{\exponentP-1}}{\sqrt{m^n}}+\sum_{\ell=0}^{n-1}\left[
\frac{32\sqrt{\exponentP-1}
c^2T^{1-\frac{1}{3\exponentP}}}{\sqrt{m^{n-\ell-1}}}\
\left[\int_{t}^{T}
\threenorm{U^{0}_{\ell,m}-\mathfrak{u}}_s^{3\exponentP}\right]^{\frac{1}{3\exponentP}}
\right].
\end{align}
\item \label{a43}For all
$n,m\in \N$, $t\in [0,T)$ we have that
\begin{align}
\threenorm{U^0_{n,m}-\mathfrak{u}}_t
\leq 
32^n(\exponentP-1)^{\frac{n}{2}}e^{nc^2T}
\exp\! \left(\frac{m^{1.5\exponentP}}{3\exponentP}\right)
m^{-n/2}.
\end{align} 
\item \label{b40} For all
$n,m\in \N$, $t\in [0,T)$, $\nu\in [0,d]\cap\Z$ we have that
\begin{align}
\Lambda_\nu(T-t)\left\lVert\pr_{\nu}\!\left(U^0_{n,m}(t,x)-\mathfrak{u}(t,x)\right)\right\rVert_\exponentP\leq 
32^n(\exponentP-1)^{\frac{n}{2}}e^{nc^2T}
\exp\! \left(\frac{m^{1.5\exponentP}}{3\exponentP}\right)
m^{-n/2} V^\exponentFirstNorm(t,x).\label{b40b}
\end{align}
\end{enumerate}
\end{proposition}
\begin{proof}
[Proof of \cref{a20}]First, \eqref{a16} follows from \cref{a02}. Next, for every 
random variable $\mathfrak{X}\colon \Omega\to \R$ with $\E [\lvert \mathfrak{X}\rvert]<\infty$ let $\var_{\exponentP}(\mathfrak{X})\in[0,\infty]$ satisfy that
$
\var_{\exponentP}(\mathfrak{X})=\lVert\mathfrak{X}-\E[\mathfrak{X}]\rVert_{\exponentP}^2.
$
Then
the Marcinkiewicz-Zygmund inequality (see \cite[Theorem~2.1]{Rio2009}), the fact that $\exponentP\in [2,\infty)$, the triangle inequality, and  Jensen's inequality show that for all $n\in\N$ and all identically distributed and independent random variables $\mathfrak{X}_k$, $k\in [1,n]\cap\Z$, with $\E [\lvert\mathfrak{X}_1\rvert]<\infty$ it holds that
\begin{align} \begin{split} 
\left(\var_{\exponentP}\left[\frac{1}{n}\sum_{k=1}^{n}\mathfrak{X}_k\right]\right)^{1/2}&=\frac{1}{n}
\left\lVert \sum_{k=1}^{n}(\mathfrak{X}_k-\E[\mathfrak{X}_k])\right\rVert_{\exponentP}\leq \frac{\sqrt{\exponentP-1}}{n}
\left(\sum_{k=1}^{n}\left\lVert \mathfrak{X}_k-\E[\mathfrak{X}_k]\right\rVert_\exponentP^2\right)^{\frac{1}{2}}\\&\leq  \frac{2\sqrt{\exponentP-1}\lVert \mathfrak{X}_1\rVert_{\exponentP}}{\sqrt{n}}.\end{split}\label{c60}
\end{align}
Next, \cite[Lemma~2.6]{NNW2023} and the assumption of \cref{a20} show \eqref{a17c}
and imply that 
for all $t\in [0,T)$ we have that
\begin{align}\label{a22}
\max_{\nu\in [0,d]\cap\Z}\sup_{y\in \R^d}
\left[\Lambda_\nu(T-t)\frac{\lvert\pr_\nu(\mathfrak{u}(t,y))\rvert}{V(t,y)}\right]
\leq 
6c e^{86c^6T^2(T-t)}.
\end{align}
Thus, the fact that $\max \{c,6e^{86c^6T^3}\}\leq V$ implies for all 
$\nu\in[0,d]\cap\Z$, $t\in [0,T)$, $y\in \R^d$ that
$\Lambda_\nu(T-t)\lvert\pr_\nu(\mathfrak{u}(t,y))\rvert\leq V^3(t,y)$. This, \eqref{a25}, and the fact that $\exponentFirstNorm\geq 3$ prove for all $t\in [0,T)$ that 
\begin{align}
\threenorm{\mathfrak{u}}_t\leq 1.\label{a27}
\end{align}
Next, \eqref{a04e}, the fact that $\exponentV\geq \exponentP$, Jensen's inequality, and \eqref{a02e} show for all $t\in [0,T)$, $x\in \R^d$ that 
\begin{align}
\left\lVert g(\mathcal{X}^{0,t,x}_T )\right\rVert_\exponentP\leq\xeqref{a04e} 
\left\lVert V(T,\mathcal{X}^{0,t,x}_T )\right\rVert_\exponentP
\leq \left\lVert V(T,\mathcal{X}^{0,t,x}_T )\right\rVert_{\exponentV}
\leq\xeqref{a02e} V(t,x).\label{b04}
\end{align}
Furthermore, the definition of $\Lambda$, \eqref{a19d}, H\"older's inequality,
the fact that $\frac{1}{\exponentV}+\frac{1}{\exponentX}+\frac{1}{\exponentZ}\leq \frac{1}{\exponentP}$, and \eqref{a02e} prove for all
$\nu\in [1,d]\cap\Z$, $t\in [0,T)$, $x\in \R^d$ that
\begin{align}
&
\left\lVert\Lambda_\nu(T-t)
(g(\mathcal{X}^{0,t,x}_T )-g(x))\pr_\nu(\mathcal{Z}^{0,t,x}_{T})
\right\rVert_\exponentP\nonumber\\
&\leq 
\left\lVert\sqrt{T-t}\xeqref{a19d}
\frac{V(T,\mathcal{X}^{0,t,x}_T)+V(T,x)}{2}
\frac{\left\lVert
\mathcal{X}^{0,t,x}_T -x\right\rVert}{\sqrt{T}}\pr_\nu(\mathcal{Z}^{0,t,x}_{T})
\right\rVert_\exponentP\nonumber\\
&\leq \sqrt{T-t}
\frac{\left\lVert V(T,\mathcal{X}^{0,t,x}_T)\right\rVert_{\exponentV}+V(t,x)}{2}
\frac{\left\lVert\left\lVert
\mathcal{X}^{0,t,x}_T -x\right\rVert\right\rVert_{\exponentX}}{\sqrt{T}}
\left\lVert\pr_\nu(\mathcal{Z}^{0,t,x}_{T})\right\rVert_{\exponentZ}\nonumber\\
&\leq \sqrt{T-t}V(t,x)\frac{V(t,x)\sqrt{T-t}}{\sqrt{T}}
\frac{c}{\sqrt{T-t}}\nonumber\\
&\leq V^3(t,x).\label{b16}
\end{align}
In addition, \eqref{b03}, the triangle inequality,  \eqref{b04},
the independence and distributional properties (cf. \cref{a02}), and \eqref{c60} imply for all 
$m,n\in \N$, $\theta\in \Theta$,
$t\in [0,T)$, $x\in \R^d$ that 
\begin{align}
&
\left\lVert\Lambda_0(T-t)
\pr_0 \left(
(g(x),0)+\sum_{i=1}^{m^n}
\frac{g(\mathcal{X}^{(\theta,0,-i),t,x}_T )-g(x) }{m^n}\mathcal{Z}^{(\theta,0,-i),t,x}_{T}
\right)\right\rVert_\exponentP\nonumber\\
&=\left\lVert
\sum_{i=1}^{m^n}
\frac{g(\mathcal{X}^{(\theta,0,-i),t,x}_T )}{m^n}
\right\rVert_\exponentP\leq \left\lVert g(\mathcal{X}^{0,t,x}_T )\right\rVert_\exponentP\leq V(t,x)\label{b17}
\end{align}
and 
\begin{align}
&\left(
\var_{\exponentP}\! \left[\Lambda_0(T-t)
\pr_0 \left(
(g(x),0)+\sum_{i=1}^{m^n}
\frac{g(\mathcal{X}^{(\theta,0,-i),t,x}_T )-g(x) }{m^n}\mathcal{Z}^{(\theta,0,-i),t,x}_{T}
\right)\right]\right)^\frac{1}{2}\nonumber\\
&=\left(\var_{\exponentP}\!
\left[
\sum_{i=1}^{m^n}
\frac{g(\mathcal{X}^{(\theta,0,-i),t,x}_T )}{m^n}
\right]
\right)^\frac{1}{2}
\leq 2\sqrt{\exponentP-1} 
\frac{\left\lVert g(\mathcal{X}^{0,t,x}_T )\right\rVert_2}{\sqrt{m^n}}\leq
2\sqrt{\exponentP-1} \frac{V(t,x)}{\sqrt{m^n}}.\label{b18}
\end{align}
Next, \eqref{b03}, the triangle inequality, \eqref{b16}, the independence and distributional properties (cf. \cref{a02}), and \eqref{c60} show for all
$m,n\in \N$, $\theta\in \Theta$,
$t\in [0,T)$, $x\in \R^d$, $\nu\in [1,d]\cap\Z$ that
\begin{align}
&\left\lVert\Lambda_\nu(T-t)
\pr_\nu \left(
(g(x),0)+\sum_{i=1}^{m^n}
\frac{g(\mathcal{X}^{(\theta,0,-i),t,x}_T )-g(x) }{m^n}\mathcal{Z}^{(\theta,0,-i),t,x}_{T}
\right)\right\rVert_{\exponentP}\nonumber\\&=\left\lVert\Lambda_\nu(T-t)\sum_{i=1}^{m^n}
\frac{g(\mathcal{X}^{(\theta,0,-i),t,x}_T )-g(x) }{m^n}\pr_\nu(\mathcal{Z}^{(\theta,0,-i),t,x}_{T})\right\rVert_{\exponentP}\nonumber\\
&\leq \left\lVert\Lambda_\nu(T-t)
(g(\mathcal{X}^{0,t,x}_T )-g(x))\pr_\nu(\mathcal{Z}^{0,t,x}_{T})
\right\rVert_\exponentP\leq V^3(t,x)\label{a19a}
\end{align}
and
\begin{align}
&\left(\var_\exponentP\!\left[\Lambda_\nu(T-t)
\pr_\nu \left(
(g(x),0)+\sum_{i=1}^{m^n}
\frac{g(\mathcal{X}^{(\theta,0,-i),t,x}_T )-g(x) }{m^n}\mathcal{Z}^{(\theta,0,-i),t,x}_{T}
\right)\right]\right)^\frac{1}{2}\nonumber\\
&=\left(\var_\exponentP\!\left[\Lambda_\nu(T-t)\sum_{i=1}^{m^n}
\frac{g(\mathcal{X}^{(\theta,0,-i),t,x}_T )-g(x) }{m^n}\pr_\nu(\mathcal{Z}^{(\theta,0,-i),t,x}_{T})\right]\right)^\frac{1}{2}\nonumber\\
&\leq 2\sqrt{\exponentP-1}\frac{\left\lVert\Lambda_\nu(T-t)
(g(\mathcal{X}^{0,t,x}_T )-g(x))\pr_\nu(\mathcal{Z}^{0,t,x}_{T})
\right\rVert_\exponentP}{\sqrt{m^n}}\leq 2\sqrt{\exponentP-1}\frac{ V^3(t,x)}{\sqrt{m^n}}.\label{a20a}
\end{align}
This, \eqref{b17}, \eqref{b18}, and the fact that $1\leq V$ imply 
for all
$m,n\in \N$, $\theta\in \Theta$,
$t\in [0,T)$, $x\in \R^d$, $\nu\in [0,d]\cap\Z$ that

\begin{align}
&\left\lVert\Lambda_\nu(T-t)
\pr_\nu \left(
(g(x),0)+\sum_{i=1}^{m^n}
\frac{g(\mathcal{X}^{(\theta,0,-i),t,x}_T )-g(x) }{m^n}\mathcal{Z}^{(\theta,0,-i),t,x}_{T}
\right)\right\rVert_{\exponentP}\leq V^3(t,x)\label{a19b}
\end{align}
and
\begin{align}
&\left(\var_\exponentP\!\left[\Lambda_\nu(T-t)
\pr_\nu \left(
(g(x),0)+\sum_{i=1}^{m^n}
\frac{g(\mathcal{X}^{(\theta,0,-i),t,x}_T )-g(x) }{m^n}\mathcal{Z}^{(\theta,0,-i),t,x}_{T}
\right)\right]\right)^\frac{1}{2} \leq2\sqrt{\exponentP-1}\frac{ V^3(t,x)}{\sqrt{m^n}}.
\end{align}
Furthermore, H\"older's inequality, the fact that
$\frac{1}{\exponentV}+\frac{1}{\exponentZ}\leq \frac{1}{\exponentP}$,  \eqref{a02e}, and the fact that $c\leq V$ show for all
$t\in [0,T)$, 
$s\in (t,T]$, 
$x\in \R^d$, $\nu\in [0,d]\cap\Z$ that
\begin{align}
\left\lVert
V(s,
\mathcal{X}^{0,t,x}_{s})
\pr_\nu(\mathcal{Z}^{0,t,x}_{s})\right\rVert_\exponentP&\leq 
\left\lVert V(s,
\mathcal{X}^{0,t,x}_{s})\right\rVert_{\exponentV}
\left\lVert\pr_\nu(\mathcal{Z}^{0,t,x}_{s})\right\rVert_\exponentZ\nonumber\\
&
\leq \xeqref{a02e}V(t,x)\xeqref{a02e}\frac{c}{\Lambda_\nu(s-t)}
\leq \frac{V^2(t,x)}{\Lambda_\nu(s-t)}.\label{a23}
\end{align}
Next, \eqref{c14},
the substitution $s=t+(T-t)r$, $ds=(T-t)dr$, $r=0\Rightarrow s=t$, $r= \frac{b-t}{T-t}\Rightarrow s =b$,
$r=\frac{s-t}{T-t}$, $1-r=1-\frac{s-t}{T-t}=\frac{T-s}{T-t}$,
and \eqref{a98}
prove for all $t\in [0,T)$, $b\in (t,T)$ that
\begin{align}
&
\P (t+(T-t)\mathfrak{r}^0\leq b)=\P(\mathfrak{r}^0
\leq \tfrac{b-t}{T-t})=\frac{1}{\mathrm{B}(\tfrac{1}{2},\tfrac{1}{2})}\int_{0}^{\tfrac{b-t}{T-t}}\frac{dr}{\sqrt{r(1-r)}}
\nonumber\\&=\frac{1}{\mathrm{B}(\tfrac{1}{2},\tfrac{1}{2})}\int_{t}^{b}\frac{\frac{ds}{T-t}}{\sqrt{
\frac{s-t}{T-t}\frac{T-s}{T-t}} }=\frac{1}{\mathrm{B}(\tfrac{1}{2},\tfrac{1}{2})}\int_{t}^{b}\frac{ds}{\sqrt{(T-s)(s-t)}}
=\int_{t}^{b}\varrho(t,s)\,ds.
\end{align}
This shows for all $t\in [0,T)$ and all measurable functions $h\colon(t,T)\to [0,\infty) $
that \begin{align}
\E\!\left [h(t+(T-t)\mathfrak{r}^0)\right]=\int_t^Th(s)\varrho(t,s)\,ds.\label{a24}
\end{align} 
Hence, the independence and distributional properties (cf. \cref{a02}), the disintegration theorem, \eqref{a04e}, \eqref{a23}, \eqref{a98}, the fact that
$\forall\,t\in [0,T)\colon \int_{t}^{T}\frac{dr}{\sqrt{r-t}}=2\sqrt{r-t}|_{r=t}^T=2\sqrt{T-t}$, the fact that 
${2\mathrm{B}(\tfrac{1}{2},\tfrac{1}{2})}\leq 7$
imply for all $t\in [0,T)$, 
$x\in \R^d$, $\nu\in [0,d]\cap\Z$ that
\begin{align}
&\left\lVert \frac{\Lambda_\nu(T-t)\left(F(0)\right)\!\left(t+(T-t) \mathfrak{r}^{0},
\mathcal{X}^{0,t,x}_{t+(T-t) \mathfrak{r}^{0}}\right)
\pr_\nu\! \left(\mathcal{Z}^{0,t,x}_{t+(T-t) \mathfrak{r}^{0}}\right)
}{\varrho(t,t+(T-t)\mathfrak{r}^{0})}
\right\rVert_\exponentP\nonumber\\
&=\xeqref{a24}\left[
\int_{t}^{T}\left\lVert \frac{\Lambda_\nu(T-t)\left(F(0)\right)\!\left(s,
\mathcal{X}^{0,t,x}_{s}\right)
\pr_{\nu}(\mathcal{Z}^{0,t,x}_{s})}{\varrho(t,s)}
\right\rVert_\exponentP^\exponentP\varrho(t,s)\,ds\right]^\frac{1}{\exponentP}\nonumber\\
&=\left[
\int_{t}^{T}\frac{\left\lVert \Lambda_\nu(T-t)\left(F(0)\right)\!\left(s,
\mathcal{X}^{0,t,x}_{s}\right)
\pr_\nu(\mathcal{Z}^{0,t,x}_{s})
\right\rVert_\exponentP^\exponentP}{(\varrho(t,s))^{\exponentP-1}}\,ds\right]^\frac{1}{\exponentP}\nonumber\\
&\leq \xeqref{a04e}
\left[
\int_{t}^{T}\frac{\left\lVert \Lambda_\nu(T-t)\frac{1}{T}V\!\left(s,
\mathcal{X}^{0,t,x}_{s}\right)
\pr_\nu(\mathcal{Z}^{0,t,x}_{s})
\right\rVert_\exponentP^\exponentP}{(\varrho(t,s))^{\exponentP-1}}\,ds\right]^\frac{1}{\exponentP}\nonumber\\
&\leq\xeqref{a23}
\left[
\int_{t}^{T}\frac{ \Lambda_\nu^\exponentP(T-t)\frac{1}{T^\exponentP}\frac{V^{2\exponentP}(t,x)}{\Lambda_\nu^\exponentP(s-t)}
}{(\varrho(t,s))^{\exponentP-1}}\,ds\right]^\frac{1}{\exponentP}\nonumber\\ 
&\leq\xeqref{a98}
\left[
\int_{t}^{T}\frac{ \frac{1}{T^\exponentP}V^{2\exponentP}(t,x)\left(\frac{T-t}{s-t}\right)^\frac{\exponentP}{2}
}{\left(\frac{1}{\mathrm{B}(\tfrac{1}{2},\tfrac{1}{2})}\right)^{\exponentP-1}\left(\frac{1}{(T-s)(s-t)}\right)^\frac{\exponentP-1}{2}}\,ds\right]^\frac{1}{\exponentP}\nonumber\\ 
&=\frac{V^2(t,x)}{T}\left(\mathrm{B}(\tfrac{1}{2},\tfrac{1}{2})\right)^{\frac{\exponentP-1}{\exponentP}}
\left[
\int_{t}^{T}
\left(\frac{T-t}{s-t}\right)^\frac{\exponentP}{2}
(T-s)^\frac{\exponentP-1}{2}(s-t)^\frac{\exponentP-1}{2}
\,ds
\right]^\frac{1}{\exponentP}\nonumber\\ 
&\leq\frac{V^2(t,x)}{T}\left(\mathrm{B}(\tfrac{1}{2},\tfrac{1}{2})\right)^{\frac{\exponentP-1}{\exponentP}}
\left[T^{\exponentP-\frac{1}{2}}
\int_{t}^{T}(s-t)^{-\frac{1}{2}}ds
\right]^\frac{1}{\exponentP}\nonumber\\ 
&=\frac{V^2(t,x)}{T}\left(\mathrm{B}(\tfrac{1}{2},\tfrac{1}{2})\right)^{\frac{\exponentP-1}{\exponentP}}
\left[T^{\exponentP-\frac{1}{2}}
2T^\frac{1}{2}
\right]^\frac{1}{\exponentP}\nonumber\\ 
&
\leq 7V^2(t,x).\label{a26}
\end{align}
Next,  \eqref{a24},
the independence and distributional properties (cf. \cref{a02}), the disintegration theorem,
\eqref{a01e}, the triangle inequality, \eqref{a25},  \eqref{a23}, and \eqref{a98}
prove for all
$\nu\in [0,d]\cap\Z$, $j\in \N_0$, $m\in \N$, $t\in [0,T)$, $x\in \R^d$ that
\begin{align}
&\left\lVert \frac{\Lambda_\nu(T-t)\left(F(U^{0}_{j,m})-
F(\mathfrak{u})\right)\!\left(t+(T-t) \mathfrak{r}^{0},
\mathcal{X}^{0,t,x}_{t+(T-t) \mathfrak{r}^{0}}\right)
\pr_\nu\! \left(\mathcal{Z}^{0,t,x}_{t+(T-t) \mathfrak{r}^{0}}\right)
}{\varrho(t,t+(T-t)\mathfrak{r}^{0})}
\right\rVert_{\exponentP}\nonumber\\
&=\xeqref{a24}\left[
\int_{t}^{T}\left\lVert \frac{\Lambda_\nu(T-t)
\left(F(U^{0}_{j,m})-F(\mathfrak{u})\right)\!\left(s,
\mathcal{X}^{0,t,x}_{s}\right)
\pr_{\nu}(\mathcal{Z}^{0,t,x}_{s})}{\varrho(t,s)}
\right\rVert_\exponentP^\exponentP\varrho(t,s)\,ds\right]^\frac{1}{\exponentP}\nonumber\\
&=\left[
\int_{t}^{T}\frac{\left\lVert \Lambda_\nu(T-t)\left(F(U^{0}_{j,m})-
F(\mathfrak{u})\right)\!\left(s,
\mathcal{X}^{0,t,x}_{s}\right)
\pr_\nu(\mathcal{Z}^{0,t,x}_{s})
\right\rVert_\exponentP^\exponentP}{(\varrho(t,s))^{\exponentP-1}}\,ds\right]^\frac{1}{\exponentP}\nonumber\\
&
\leq \xeqref{a01e}\left[
\int_{t}^{T}\frac{\left\lVert \Lambda_\nu(T-t)
\sum_{i=0}^{d}L_i\Lambda_i(T)\left\lvert
(U^{0}_{j,m}-\mathfrak{u})(s,
\mathcal{X}^{0,t,x}_{s})
\pr_\nu(\mathcal{Z}^{0,t,x}_{s})
\right\rvert\right\rVert_\exponentP^\exponentP}{(\varrho(t,s))^{\exponentP-1}}\,ds\right]^\frac{1}{\exponentP}\nonumber\\
&\leq \left[\int_{t}^{T}\frac{ \left[\Lambda_\nu(T-t)
\sum_{i=0}^{d}L_i\Lambda_i(T)\left\lVert
(U^{0}_{j,m}-\mathfrak{u})(s,
\mathcal{X}^{0,t,x}_{s})\pr_\nu(\mathcal{Z}^{0,t,x}_{s})
\right\rVert_\exponentP\right]^\exponentP}{(\varrho(t,s))^{\exponentP-1}}\,ds\right] ^\frac{1}{\exponentP}\nonumber\\
&\leq \left[\int_{t}^{T}\frac{\left[ \Lambda_\nu(T-t)
\sum_{i=0}^{d}L_i\frac{\sqrt{T}}{\sqrt{T-s}}\left\lVert
\Lambda_i(T-s)
(U^{0}_{j,m}-\mathfrak{u})(s,
\mathcal{X}^{0,t,x}_{s})\pr_\nu(\mathcal{Z}^{0,t,x}_{s})
\right\rVert_\exponentP\right]^\exponentP}{(\varrho(t,s))^{\exponentP-1}}\,ds\right]^\frac{1}{\exponentP}\nonumber\\
&\leq \left[\int_{t}^{T}\frac{\left[ \Lambda_\nu(T-t)
\sum_{i=0}^{d}L_i\frac{\sqrt{T}}{\sqrt{T-s}}\left\lVert
\left\lVert
\Lambda_i(T-s)
(U^{0}_{j,m}-\mathfrak{u})(s,
\tilde{x})\pr_\nu(\tilde{z})
\right\rVert_\exponentP
\Bigr|_{\tilde{x}=\mathcal{X}^{0,t,x}_{s}, \tilde{z}=\mathcal{Z}^{0,t,x}_{s}}
\right\rVert_\exponentP\right]^\exponentP}{(\varrho(t,s))^{\exponentP-1}}\,ds\right]^\frac{1}{\exponentP}\nonumber\\
&\leq\xeqref{a25} \left[\int_{t}^{T}\frac{\left[ \Lambda_\nu(T-t)
c\frac{\sqrt{T}}{\sqrt{T-s}}
\threenorm{U^{0}_{j,m}-\mathfrak{u}}_s\left\lVert
V^{\exponentFirstNorm}(s,
\mathcal{X}^{0,t,x}_{s})
\pr_\nu(\mathcal{Z}^{0,t,x}_{s})
\right\rVert_\exponentP\right]^\exponentP}{(\varrho(t,s))^{\exponentP-1}}\,ds\right]^\frac{1}{\exponentP}\nonumber\\
&\leq\xeqref{a23} \left[\int_{t}^{T}\frac{\left[ \Lambda_\nu(T-t)
c\frac{\sqrt{T}}{\sqrt{T-s}}
\threenorm{U^{0}_{j,m}-\mathfrak{u}}_s
V^{\exponentFirstNorm}(t,x)\frac{c}{\Lambda_\nu(s-t)}
\right]^\exponentP
}{(\varrho(t,s))^{\exponentP-1}}\,ds\right]^\frac{1}{\exponentP}\nonumber\\
&\leq \xeqref{a98}\left[\int_{t}^{T}\frac{\left[ 
c^2
\threenorm{U^{0}_{j,m}-\mathfrak{u}}_s
V^{\exponentFirstNorm}(t,x)\frac{\sqrt{T}}{\sqrt{T-s}}\frac{\sqrt{T-t}}{\sqrt{s-t}}
\right]^\exponentP
}{\left(\frac{1}{\mathrm{B}(\tfrac{1}{2},\tfrac{1}{2})}\right)^{\exponentP-1}\left(\frac{1}{\sqrt{(T-s)(s-t)}}\right)^{\exponentP-1}}\,ds \right]^\frac{1}{\exponentP}\nonumber\\
&=
c^2
\left(\mathrm{B}(\tfrac{1}{2},\tfrac{1}{2})\right)^{\frac{\exponentP-1}{\exponentP}}
 \left[
\int_{t}^{T}\frac{
\left[ 
\threenorm{U^{0}_{j,m}-\mathfrak{u}}_s
V^{\exponentFirstNorm}(t,x)\sqrt{T}\sqrt{T-t}
\right]^\exponentP
}{\sqrt{(T-s)(s-t)}}\,ds\right]^\frac{1}{\exponentP}\nonumber\\
&\leq 
4
c^2\sqrt{T(T-t)}V^\exponentFirstNorm(t,x)
\left[
\int_{t}^{T}\frac{
\threenorm{U^{0}_{j,m}-\mathfrak{u}}_s^\exponentP
}{\sqrt{(T-s)(s-t)}}\,ds\right]^\frac{1}{\exponentP}.\label{a31}
\end{align}
Furthermore,
the substitution $s=t+(T-t)r$, $ds=(T-t)dr$, $ s=t\Rightarrow r=0$,
$s=T\Rightarrow r=1$,
$T-s=T-t-(T-t)r=(T-t)(1-r)$,
$s-t=(T-t)r$ and the definition of the Beta function imply for all $t\in [0,T)$ that
\begin{align}
\int_{t}^{T}\frac{ds}{(T-s)^{\frac{3}{4}}(s-t)^{\frac{3}{4}}}=
\int_{0}^{1}\frac{(T-t)dr}{[(T-t)(1-r)]^\frac{3}{4}
[(T-t)r)]^\frac{3}{4}
}=(T-t)^{-\frac{1}{2}}\mathrm{B}(\tfrac{1}{4},\tfrac{1}{4}).
\label{a32}
\end{align}
Therefore, \eqref{a31}, H\"older's inequality, the fact that 
$\frac{2}{3}+\frac{1}{3}=1$, and the fact that
$(\mathrm{B}(\frac{1}{4},\frac{1}{4}))^\frac{1}{3}\leq 2$
show for all
$\nu\in [0,d]\cap\Z$, $j\in \N_0$, $m\in \N$, $t\in [0,T)$, $x\in \R^d$ that
\begin{align}
&\left\lVert \frac{\Lambda_\nu(T-t)\left(F(U^{0}_{j,m})-F(\mathfrak{u})\right)\!\left(t+(T-t) \mathfrak{r}^{0},
\mathcal{X}^{0,t,x}_{t+(T-t) \mathfrak{r}^{0}}\right)
\pr_\nu\! \left(\mathcal{Z}^{0,t,x}_{t+(T-t) \mathfrak{r}^{0}}\right)
}{\varrho(t,t+(T-t)\mathfrak{r}^{0})}
\right\rVert_\exponentP\nonumber\\
&\leq \xeqref{a31}
4c^2\sqrt{T}\sqrt{T-t}V^\exponentFirstNorm(t,x)
\left[\int_{t}^{T}\frac{ds}{(T-s)^{\frac{1}{2}\frac{3}{2}}(s-t)^{\frac{1}{2}\frac{3}{2}}}\right]^{\frac{2}{3}\frac{1}{\exponentP}}
\left[\int_{t}^{T}
\threenorm{U^{0}_{j,m}-\mathfrak{u}}_s^{\exponentP\cdot 3}\right]^{\frac{1}{3}\frac{1}{\exponentP}}\nonumber\\
&\leq 4
c^2\sqrt{T}\sqrt{T-t}V^\exponentFirstNorm(t,x)
\left[\int_{t}^{T}\frac{ds}{(T-s)^{\frac{3}{4}}(s-t)^{\frac{3}{4}}}\right]^{\frac{2}{3\exponentP}}
\left[\int_{t}^{T}
\threenorm{U^{0}_{j,m}-\mathfrak{u}}_s^{3\exponentP}\right]^{\frac{1}{3\exponentP}}\nonumber\\
&\leq 4
c^2\sqrt{T}\sqrt{T-t}V^\exponentFirstNorm(t,x)
\left[\xeqref{a32}
(T-t)^{-\frac{1}{2}}
\mathrm{B}(\tfrac{1}{4},\tfrac{1}{4})\right]^\frac{2}{3\exponentP}
\left[\int_{t}^{T}
\threenorm{U^{0}_{j,m}-\mathfrak{u}}_s^{3\exponentP}\right]^{\frac{1}{3\exponentP}}
\nonumber\\
&\leq 8
c^2T^{1-\frac{1}{3\exponentP}}V^\exponentFirstNorm(t,x)
\left[\int_{t}^{T}
\threenorm{U^{0}_{j,m}-\mathfrak{u}}_s^{3\exponentP}\right]^{\frac{1}{3\exponentP}}.\label{a29}
\end{align}
This, the triangle inequality, and the distributional and independence properties (cf. \cref{a02})
prove for all 
$\nu\in [0,d]\cap\Z$,  $\ell,m\in \N$, $t\in [0,T)$, $x\in \R^d$ that
\begin{align}
&\left\lVert \frac{\Lambda_\nu(T-t)\left(F(U^{0}_{\ell,m})-F(U^{1}_{\ell-1,m})\right)\!\left(t+(T-t) \mathfrak{r}^{0},
\mathcal{X}^{0,t,x}_{t+(T-t) \mathfrak{r}^{0}}\right)
\pr_\nu\! \left(\mathcal{Z}^{0,t,x}_{t+(T-t) \mathfrak{r}^{0}}\right)
}{\varrho(t,t+(T-t)\mathfrak{r}^{0})}
\right\rVert_\exponentP\nonumber\\
&\leq \sum_{j=\ell-1}^{\ell}\left\lVert \frac{\Lambda_\nu(T-t)\left(F(U^{0}_{j,m})-F(\mathfrak{u})\right)\!\left(t+(T-t) \mathfrak{r}^{0},
\mathcal{X}^{0,t,x}_{t+(T-t) \mathfrak{r}^{0}}\right)
\pr_\nu\! \left(\mathcal{Z}^{0,t,x}_{t+(T-t) \mathfrak{r}^{0}}\right)
}{\varrho(t,t+(T-t)\mathfrak{r}^{0})}
\right\rVert_\exponentP\nonumber\\
&\leq 8
c^2T^{1-\frac{1}{3\exponentP}}V^\exponentFirstNorm(t,x)
\sum_{j=\ell-1}^{\ell}
\left[\int_{t}^{T}
\threenorm{U^{0}_{j,m}-\mathfrak{u}}_s^{3\exponentP}\right]^{\frac{1}{3\exponentP}}.\label{c15}
\end{align}
Hence, 
\eqref{a04d},
the fact that $\forall\,\theta\in \Theta,m\in\N\colon U^\theta_{0,m}=0$, \eqref{a19b}, \eqref{a26},  \eqref{a27}, 
the independence and distributional properties (cf. \cref{a02}),
and an induction argument prove
for all $n,m\in \N$, $\ell\in [0,n-1]\cap\Z$, $\nu \in [0,d]\cap\Z$, $t\in [0,T)$, $x\in \R^d$ that
\begin{align}
&\left(\sup_{s\in [0,T)}\threenorm{U_{n,m}^0}_s\right)+\left\lVert \frac{\Lambda_\nu(T-t)(F(U^{0}_{\ell,m}))\!\left(t+(T-t) \mathfrak{r}^{0},
X^{0,t,x}_{t+(T-t) \mathfrak{r}^{0}}\right)
\pr_\nu\! \left(Z^{0,t,x}_{t+(T-t) \mathfrak{r}^{0}}\right)
}{\varrho(t,t+(T-t)\mathfrak{r}^{0})}
\right\rVert_\exponentP<\infty.
\end{align}
This, linearity of the expectation, and the independence and distributional properties (cf. \cref{a02}) imply for all 
$n,m\in \N$, $t\in [0,T)$, $x\in \R^d$ that
\begin{align}
&
\E\!\left[
U^{0}_{n,m}(t,x)\right]\nonumber\\
&= (g(x),0)+\sum_{i=1}^{m^n}
\E\!\left[
\frac{g(\mathcal{X}^{(0,0,-i),t,x}_T )-g(x) }{m^n}\mathcal{Z}^{(0,0,-i),t,x}_{T}\right]
\nonumber\\
&\quad  +\sum_{\ell=0}^{n-1}\E\!\left[\sum_{i=1}^{m^{n-\ell}} \frac{\left(F(U^{(0,\ell,i)}_{\ell,m})-
\1_\N(\ell)
F(U^{(0,\ell,-i)}_{\ell-1,m})\right)\!\left(t+(T-t) \mathfrak{r}^{(0,\ell,i)},
\mathcal{X}^{(0,\ell,i),t,x}_{t+(T-t) \mathfrak{r}^{(0,\ell,i)}}\right)
\mathcal{Z}^{(0,\ell,i),t,x}_{t+(T-t) \mathfrak{r}^{(0,\ell,i)}}}{m^{n-\ell}\varrho(t,t+(T-t)\mathfrak{r}^{(0,\ell,i)})}\right]\nonumber\\
&=(g(x),0)+\E\!\left[(g(\mathcal{X}^{0,t,x}_T)-g(x))\mathcal{Z}^{0,t,x}_T\right]\nonumber\\
&\quad +
\sum_{\ell=0}^{n-1}\Biggl(
\E\!\left[\frac{F(U^0_{\ell,m})
(t+(T-t)\mathfrak{r}^0,\mathcal{X}^{0,t,x}_{t+(T-t)\mathfrak{r}^0} )\mathcal{Z}^{0,t,x}_{t+(T-t)\mathfrak{r}^0}}{\varrho(t,t+(T-t)\mathfrak{r}^0)}
\right]\nonumber\\
&\qquad\qquad\qquad -
\1_\N(\ell)\E\!\left[\frac{F(U^0_{\ell-1,m})
(t+(T-t)\mathfrak{r}^0,\mathcal{X}^{0,t,x}_{t+(T-t)\mathfrak{r}^0} )\mathcal{Z}^{0,t,x}_{t+(T-t)\mathfrak{r}^0}}{\varrho(t,t+(T-t)\mathfrak{r}^0)}
\right]\Biggr)\nonumber\\
&=\E\!\left[(g(\mathcal{X}^{0,t,x}_T)\mathcal{Z}^{0,t,x}_T\right]
+
\E\!\left[\frac{F(U^0_{n-1,m})
(t+(T-t)\mathfrak{r}^0,\mathcal{X}^{0,t,x}_{t+(T-t)\mathfrak{r}^0} )\mathcal{Z}^{0,t,x}_{t+(T-t)\mathfrak{r}^0}}{\varrho(t,t+(T-t)\mathfrak{r}^0)}
\right].\label{a33}
\end{align}
Next, 
\eqref{a05e}, \eqref{a24}, the disintegration theorem, and
the independence and distributional properties (cf. \cref{a02}) prove for all
$t\in [0,T)$, $x\in \R^d$ that
\begin{align}
\mathfrak{u}(t,x)&
=\xeqref{a05e}\E\!\left [g(\mathcal{X}^{0,t,x}_{T} )\mathcal{Z}^{0,t,x}_{T} \right] + \int_{t}^{T}
\E \!\left[\frac{
f(r,\mathcal{X}^{0,t,x}_{r},\mathfrak{u}(r,\mathcal{X}^{0,t,x}_{r}))\mathcal{Z}^{0,t,x}_{r}}{\varrho(t,r)}\right]\varrho(t,r)\,dr\nonumber
\\
&=\xeqref{a24}\E\!\left[(g(\mathcal{X}^{0,t,x}_T)\mathcal{Z}^{0,t,x}_T\right]
+
\E\!\left[\frac{F(\mathfrak{u})
(t+(T-t)\mathfrak{r}^0,\mathcal{X}^{0,t,x}_{t+(T-t)\mathfrak{r}^0} )\mathcal{Z}^{0,t,x}_{t+(T-t)\mathfrak{r}^0}}{\varrho(t,t+(T-t)\mathfrak{r}^0)}
\right].
\end{align}
This, \eqref{a33}, Jensen's inequality, and \eqref{a29} 
imply for all
$t\in [0,T)$, $x\in \R^d$, $\nu\in [0,d]\cap\Z$ that
\begin{align}
&
\Lambda_\nu(T-t)\left\lvert\pr_{\nu}\!\left(
\E\!\left[
U^{0}_{n,m}(t,x)\right]-\mathfrak{u}(t,x)\right)\right\rvert\nonumber\\
&=
\left\lvert
\Lambda_\nu(T-t)
\E\!\left[
\frac{(F(U^0_{n-1,m})-F(\mathfrak{u}))
(t+(T-t)\mathfrak{r}^0,\mathcal{X}^{0,t,x}_{t+(T-t)\mathfrak{r}^0} )\mathcal{Z}^{0,t,x}_{t+(T-t)\mathfrak{r}^0}}{\varrho(t,t+(T-t)\mathfrak{r}^0)}\right]
\right\rvert\nonumber\\
&
\leq 
\left\lVert
\Lambda_\nu(T-t)
\frac{(F(U^0_{n-1,m})-F(\mathfrak{u}))
(t+(T-t)\mathfrak{r}^0,\mathcal{X}^{0,t,x}_{t+(T-t)\mathfrak{r}^0} )\mathcal{Z}^{0,t,x}_{t+(T-t)\mathfrak{r}^0}}{\varrho(t,t+(T-t)\mathfrak{r}^0)}
\right\rVert_\exponentP\nonumber\\
&\leq 8
c^2T^{1-\frac{1}{3\exponentP}}V^\exponentFirstNorm(t,x)
\left[\int_{t}^{T}
\threenorm{U^{0}_{n-1,m}-\mathfrak{u}}_s^{3\exponentP}\right]^{\frac{1}{3\exponentP}}.\label{a29b}
\end{align}
Furthermore,
\eqref{a04d}, the triangle inequality, the independence and distributional properties (cf. \cref{a02}), \eqref{a20a},
\eqref{c60},
 \eqref{a26},  \eqref{c15}, and the fact that $1\leq V$ prove for all
$t\in [0,T)$, $x\in \R^d$, $\nu\in [0,d]\cap\Z$ that
\begin{align} 
&
\Lambda_\nu(T-t)\left(\var_\exponentP\!\left[
\pr_\nu (U^{0}_{n,m}(t,x))\right]\right)^\frac{1}{2}\nonumber\\
&= 
\Lambda_\nu(T-t)\left(\var_\exponentP\!\left[
\pr_{\nu}\left(
(g(x),0)+\sum_{i=1}^{m^n}
\tfrac{g(\mathcal{X}^{(0,0,-i),t,x}_T )-g(x) }{m^n}\mathcal{Z}^{(0,0,-i),t,x}_{T}\right)
\right]\right)^\frac{1}{2}\nonumber
\\
& \quad +\sum_{\ell=0}^{n-1}
\left(\var_\exponentP\!\left[
\sum_{i=1}^{m^{n-\ell}} \tfrac{\Lambda_\nu(T-t)\left(F(U^{(0,\ell,i)}_{\ell,m})-
\1_\N(\ell)
F(U^{(0,\ell,-i)}_{\ell-1,m})\right)\!\left(t+(T-t) \mathfrak{r}^{(0,\ell,i)},
\mathcal{X}^{(0,\ell,i),t,x}_{t+(T-t) \mathfrak{r}^{(0,\ell,i)}}\right)
\pr_{\nu}\left(\mathcal{Z}^{(0,\ell,i),t,x}_{t+(T-t) \mathfrak{r}^{(0,\ell,i)}}\right)
}{m^{n-\ell}\varrho(t,t+(T-t)\mathfrak{r}^{(0,\ell,i)})}\right]
\right)^\frac{1}{2}
\nonumber\\
&\leq \xeqref{a20a}\tfrac{2\sqrt{\exponentP-1}V^3(t,x)}{\sqrt{m^n}}
+\sum_{\ell=0}^{n-1}\tfrac{2\sqrt{\exponentP-1}}{\sqrt{m^{n-\ell}}}
\left(\var_\exponentP\!\left[\tfrac{\Lambda_\nu(T-t)\left(F(U^{0}_{\ell,m})-
\1_\N(\ell)
F(U^{1}_{\ell-1,m})\right)\!\left(t+(T-t) \mathfrak{r}^{0},
\mathcal{X}^{0,t,x}_{t+(T-t) \mathfrak{r}^{0}}\right)
\pr_{\nu}\left(\mathcal{Z}^{0,t,x}_{t+(T-t) \mathfrak{r}^{0}}\right)
}{\varrho(t,t+(T-t)\mathfrak{r}^{0})}\right]
\right)^\frac{1}{2}
\nonumber\\
&\leq \tfrac{2\sqrt{\exponentP-1}V^3(t,x)}{\sqrt{m^n}}
+
\sum_{\ell=0}^{n-1}\left[\tfrac{2\sqrt{\exponentP-1}}{\sqrt{m^{n-\ell}}}
\left\lVert\tfrac{\Lambda_\nu(T-t)\left(F(U^{0}_{\ell,m})-
\1_\N(\ell)
F(U^{1}_{\ell-1,m})\right)\!\left(t+(T-t) \mathfrak{r}^{0},
\mathcal{X}^{0,t,x}_{t+(T-t) \mathfrak{r}^{0}}\right)
\pr_{\nu}\left(\mathcal{Z}^{0,t,x}_{t+(T-t) \mathfrak{r}^{0}}\right)
}{\varrho(t,t+(T-t)\mathfrak{r}^{0})}\right\rVert_\exponentP\right]
\nonumber\\
&\leq 
\tfrac{2\sqrt{\exponentP-1}V^3(t,x)}{\sqrt{m^n}}
+
\tfrac{2\sqrt{\exponentP-1}}{\sqrt{m^n}}
\left\lVert \tfrac{\Lambda_\nu(T-t)\left(F(0)\right)\!\left(t+(T-t) \mathfrak{r}^{0},
\mathcal{X}^{0,t,x}_{t+(T-t) \mathfrak{r}^{0}}\right)
\pr_\nu\! \left(\mathcal{Z}^{0,t,x}_{t+(T-t) \mathfrak{r}^{0}}\right)
}{\varrho(t,t+(T-t)\mathfrak{r}^{0})}
\right\rVert_\exponentP\nonumber\\
&\quad +
\sum_{\ell=1}^{n-1}\left[
\tfrac{2\sqrt{\exponentP-1}}{\sqrt{m^{n-\ell}}}
\left\lVert\tfrac{\Lambda_\nu(T-t)\left(F(U^{0}_{\ell,m})-
F(U^{1}_{\ell-1,m})\right)\!\left(t+(T-t) \mathfrak{r}^{0},
\mathcal{X}^{0,t,x}_{t+(T-t) \mathfrak{r}^{0}}\right)
\pr_{\nu}\left(\mathcal{Z}^{0,t,x}_{t+(T-t) \mathfrak{r}^{0}}\right)
}{\varrho(t,t+(T-t)\mathfrak{r}^{0})}\right\rVert_\exponentP\right]\nonumber\\
&\leq \tfrac{2\sqrt{\exponentP-1}V^3(t,x)}{\sqrt{m^n}}+\xeqref{a26}\tfrac{2\sqrt{\exponentP-1}7V^2(t,x)}{\sqrt{m^n}}+\xeqref{c15}\sum_{\ell=1}^{n-1}\left[\tfrac{2\sqrt{\exponentP-1}}{\sqrt{m^{n-\ell}}}\left[
8
c^2T^{1-\frac{1}{3\exponentP}}V^\exponentFirstNorm(t,x)
\sum_{j=\ell-1}^{\ell}
\left[\int_{t}^{T}
\threenorm{U^{0}_{j,m}-\mathfrak{u}}_s^{3\exponentP}\right]^{\frac{1}{3\exponentP}}
\right]\right]\nonumber\\
&\leq \tfrac{16\sqrt{\exponentP-1}V^3(t,x)}{\sqrt{m^n}}+\sum_{\ell=1}^{n-1}\left[
\sum_{j=\ell-1}^{\ell}
\tfrac{16\sqrt{\exponentP-1}
c^2T^{1-\frac{1}{3\exponentP}}V^\exponentFirstNorm(t,x)}{\sqrt{m^{n-j-1}}}\
\left[\int_{t}^{T}
\threenorm{U^{0}_{j,m}-\mathfrak{u}}_s^{3\exponentP}\right]^{\frac{1}{3\exponentP}}
\right]\nonumber\\
&= \tfrac{16\sqrt{\exponentP-1}V^3(t,x)}{\sqrt{m^n}}+\sum_{j=0}^{n-1}\left[
\sum_{\ell\in [1,n-1]\cap \{j,j+1\}}
\tfrac{16\sqrt{\exponentP-1}
c^2T^{1-\frac{1}{3\exponentP}}V^\exponentFirstNorm(t,x)}{\sqrt{m^{n-j-1}}}\
\left[\int_{t}^{T}
\threenorm{U^{0}_{j,m}-\mathfrak{u}}_s^{3\exponentP}\right]^{\frac{1}{3\exponentP}}
\right]\nonumber\\
&= \tfrac{16\sqrt{\exponentP-1}V^3(t,x)}{\sqrt{m^n}}+\sum_{j=0}^{n-1}\left[
(2-\1_{n-1}(j))
\tfrac{16\sqrt{\exponentP-1}
c^2T^{1-\frac{1}{3\exponentP}}V^\exponentFirstNorm(t,x)}{\sqrt{m^{n-j-1}}}\
\left[\int_{t}^{T}
\threenorm{U^{0}_{j,m}-\mathfrak{u}}_s^{3\exponentP}\right]^{\frac{1}{3\exponentP}}
\right].
\label{a04f}
\end{align}
Thus, the triangle inequality, the definition of $\var_\mathfrak{p}$, \eqref{a29b}, and the fact that
$V^3\leq V^\exponentFirstNorm$
prove for all 
$\nu\in [0,d]\cap\Z$, $n,m\in \N$, $t\in [0,T)$, $x\in \R^d$ that
\begin{align}
&
\left\lVert
\Lambda_\nu(T-t)
\pr_{\nu}\!\left(
U^0_{n,m}(t,x)-\mathfrak{u}(t,x)\right)\right\rVert_\mathfrak{p}\nonumber\\
&
\leq 
\Lambda_\nu(T-t)\left(\var_\mathfrak{p}\!\left[
\pr_\nu (U^{0}_{n,m}(t,x))\right]\right)^\frac{1}{2}+
\Lambda_\nu(T-t)\left\lvert\pr_{\nu}\!\left(
\E\!\left[
U^{0}_{n,m}(t,x)\right]-\mathfrak{u}(t,x)\right)\right\rvert\nonumber\\
&\leq 
\frac{16\sqrt{\exponentP-1}V^{\exponentFirstNorm}(t,x)}{\sqrt{m^n}}+\sum_{j=0}^{n-1}\left[
\frac{32\sqrt{\exponentP-1}
c^2T^{1-\frac{1}{3\exponentP}}V^\exponentFirstNorm(t,x)}{\sqrt{m^{n-j-1}}}\
\left[\int_{t}^{T}
\threenorm{U^{0}_{j,m}-\mathfrak{u}}_s^{3\exponentP}\right]^{\frac{1}{3\exponentP}}
\right].
\end{align}
Therefore, \eqref{a25}
implies for all $n,m\in \N$, $t\in [0,T)$ that
\begin{align}
\threenorm{U^0_{n,m}-\mathfrak{u}}_t\leq 
\frac{16\sqrt{\exponentP-1}}{\sqrt{m^n}}+\sum_{\ell=0}^{n-1}\left[
\frac{32\sqrt{\exponentP-1}
c^2T^{1-\frac{1}{3\exponentP}}}{\sqrt{m^{n-\ell-1}}}\
\left[\int_{t}^{T}
\threenorm{U^{0}_{\ell,m}-\mathfrak{u}}_s^{3\exponentP}\right]^{\frac{1}{3\exponentP}}
\right].\label{a42b}
\end{align}
This shows
\eqref{a42}. 

Next, \eqref{a42b}, \cite[Lemma 3.11]{HJKN2020},   \eqref{a27}, and the fact that $1+c^2T\leq e^{c^2T}$ prove for all
$n,m\in \N$, $t\in [0,T)$ that
\begin{align}
&
\threenorm{U^0_{n,m}-\mathfrak{u}}_t\nonumber
\\
&
\leq \left[16\sqrt{\exponentP-1}+
32\sqrt{\exponentP-1}
c^2T^{1-\frac{1}{3\exponentP}}\cdot T^\frac{1}{3\exponentP}\cdot1\right]
\exp\! \left(\frac{m^{1.5\exponentP}}{3\exponentP}\right)
m^{-n/2}
\left[1+32\sqrt{\exponentP-1}
c^2T^{1-\frac{1}{3\exponentP}}\cdot T^\frac{1}{3\exponentP}\right]\nonumber\\
&=
\left[16\sqrt{\exponentP-1}+
32\sqrt{\exponentP-1}
c^2 T\right]
\exp\! \left(\frac{m^{1.5\exponentP}}{3\exponentP}\right)
m^{-n/2}
\left[1+32\sqrt{\exponentP-1}
c^2T\right]^{n-1}\nonumber\\
&\leq 
32\sqrt{\exponentP-1}(1+c^2T)
\exp\! \left(\frac{m^{1.5\exponentP}}{3\exponentP}\right)
m^{-n/2}\left[32\sqrt{\exponentP-1}(1+c^2T)\right]^{n-1}\nonumber\\
&\leq 32^n(\exponentP-1)^{\frac{n}{2}}e^{nc^2T}
\exp\! \left(\frac{m^{1.5\exponentP}}{3\exponentP}\right)
m^{-n/2}.\label{b41}
\end{align}
This shows \eqref{a43}. 

Next, \eqref{b41} and \eqref{a25} prove for all
$n,m\in \N$, $t\in [0,T)$, $\nu\in [0,d]\cap\Z$, $x\in \R^d$ that
\begin{align}
\Lambda_\nu(T-t)\left\lVert\pr_{\nu}\!\left(U^0_{n,m}(t,x)-\mathfrak{u}(t,x)\right)\right\rVert_\exponentP\leq 
 32^n(\exponentP-1)^{\frac{n}{2}}e^{nc^2T}
\exp\! \left(\frac{m^{1.5\exponentP}}{3\exponentP}\right)
m^{-n/2}V^{\exponentFirstNorm}(t,x).
\end{align}
This implies \eqref{b40} and completes the proof of \cref{a20}.
\end{proof}
\cref{s01} below is a simple calculation and is included for convenience of the reader.
\begin{lemma}\label{s01}
Let $d\in \N$, $a\in [0,\infty)$, $p\in [2 ,\infty)$,
$\varphi\in C(\R^d,\R)$ satisfy for all $x\in \R^d$ that $\varphi (x)=(a+\lVert x\rVert^2)^p$. Then 
$\Delta \varphi (x)\leq (4p^2-2p) (a+\lVert x\rVert^2)^{p-1} $.
\end{lemma}

\begin{proof}
[Proof of \cref{s01}]First, we have for all $i\in[1,d]\cap\Z$, $d\in \N$ that
$\frac{\partial \varphi}{\partial x_i}(x)= p(a+\lVert x\rVert^2)^{p-1}2x_i$ and
\begin{align} \begin{split} 
\frac{\partial^2 \varphi}{\partial x_i^2}(x)
&= p(p-1)(a+\lVert x\rVert^2)^{p-2}2x_i2x_i+2p(a+\lVert x\rVert^2)^{p-1}\\
&
=4p(p-1)(a+\lVert x\rVert^2)^{p-2}x_i^2+2p(a+\lVert x\rVert^2)^{p-1}.
\end{split}
\end{align}
Hence, we have for all $x\in \R^d$ that \begin{align} \begin{split} 
&
(\Delta \varphi )(x)= \sum_{i=1}^{d}\frac{\partial^2 \varphi}{\partial x_i^2}(x)\\&\leq 4p(p-1)(a+\lVert x\rVert^2)^{p-1}+2p(a+\lVert x\rVert^2)^{p-1}
=(4p^2-2p) (a+\lVert x\rVert^2)^{p-1} .\end{split}
\end{align}
This completes the proof of \cref{s01}.
\end{proof}

\section{$L^\mathfrak{p}$-error bound for the MLP approximations involving Brownian motions}\label{c11}
\begin{proof}
[Proof of \cref{a41}]
Let $p= \max \{4,\lceil 3\beta \exponentP\rceil\}$. 
Without lost of generality we can assume that
\begin{align}\label{a10}
\frac{2^{\frac{p}{2}}\Gamma(\tfrac{p+1}{2})}{\sqrt{\pi}}\leq c.
\end{align}
For every $d\in \N$ let
$a_d\in \R$ and
 $\varphi_d\in C(\R^d,\R)$ satisfy for all $x\in \R^d$ that
\begin{align}
\varphi _d(x)= 2^p c^p d^p (a_d+ \lVert x\rVert^2)^{p/2},\quad
a_d\geq d+2p+d^{2c}+(\max \{ c,48e^{86c^6T^3}\})^2.\label{c03}
\end{align}
Then \cref{s01} shows  for all $d\in \N$ that $\frac{1}{2}\Delta\varphi_d\leq (2p^2-p)\varphi_d$. This, \eqref{c01}, 
 and, e.g., \cite[Lemma~2.2]{CHJ2021} show for all $d\in \N$, $t\in [0,T]$, $s\in [t,T]$, $x\in \R^d$ that
\begin{align}
\E [\varphi_d(\mathcal{X}^{d,0,t,x}_s)]\leq e^{(2p^2-p)(s-t)}\varphi_d(x).
\end{align}
For every $d\in \N$ let $V_d\in C([0,T]\times \R^d,\R)$ satisfy for all $t\in [0,T]$, $x\in \R^d$ that
\begin{align}
V_d(t,x)= e^{\beta(2p-1)(T-t)}(\varphi_d(x))^{\frac{\beta}{p}}.\label{c04}
\end{align}
Then 
\begin{align}
&
\left\lVert
V_d(s,\mathcal{X}^{d,0,t,x}_s)\right\rVert_{3\mathfrak{p}}
=e^{\beta(2p-1)(T-s)}
\left\lVert
(\varphi_d(\mathcal{X}^{d,0,t,x}_s))^{\frac{\beta}{p}}
\right\rVert_{3\mathfrak{p}}\leq 
e^{\beta(2p-1)(T-s)}
\left(
\E \!\left[
(\varphi_d(\mathcal{X}^{d,0,t,x}_s))^{\frac{3\mathfrak{p}\beta}{p}}
\right]\right)^\frac{1}{3\mathfrak{p}}\nonumber\\
&\leq 
e^{\beta(2p-1)(T-s)}
\left(
\E \!\left[\varphi_d(\mathcal{X}^{d,0,t,x}_s)
\right]
\right)^{\frac{3\mathfrak{p}\beta}{p} \frac{1}{3\mathfrak{p}} }
\leq 
e^{\beta(2p-1)(T-s)}
\left(
e^{(2p^2-p)(s-t)}\varphi_d(x)
\right)^\frac{\beta}{p}\nonumber\\
&=e^{\beta(2p-1)(T-t)}(\varphi_d(x))^{\frac{\beta}{p}}
= V_d(t,x). \label{a11}
\end{align}
First, note for all $t\in (0,T]$ that $\left\lVert \frac{W(t)}{\sqrt{t}}\right\rVert^2$ is chi-squared distributed with $d$ degrees of freedom. This and Jensen's inequality show for all $t\in (0,T]$
 that
\begin{align}
\left(
\E \!\left[
\left\lVert W_t\right\rVert^{{p}}
\right]\right)^2\leq 
\E \!\left[
\left\lVert W_t\right\rVert^{2p}
\right]=(2t)^{{p}}\frac{\Gamma(0.5d+{p})}{\Gamma(0.5d)}
=(2t)^{{p}}\prod_{k=0}^{{p}-1}(0.5d+k)
\end{align} and
\begin{align}
\left\lVert
\left\lVert
W_t\right\rVert
\right\rVert_p
&=\left(
\E \!\left[
\left\lVert W_t\right\rVert^{{p}}
\right]\right)^\frac{2}{2p}
\leq (2t)^\frac{1}{2}
\left(\prod_{k=0}^{{p}-1}(0.5d+k)\right)^\frac{1}{2p}
\leq \sqrt{2t\left(\frac{d}{2}+{p}-1\right)}\leq \sqrt{t(d+2p)}.\label{a05}
\end{align}This, \eqref{c01}, Jensen's inequality,  the fact that
$3\mathfrak{p}\leq p$, \eqref{c03}, the fact that $\beta\geq1$, and \eqref{c04} show for all $d\in \N$, $t\in [0,T]$, $s\in [t,T]$, $x\in \R^d$ that
\begin{align}
\left\lVert X^{d,0,t,x}_s-x\right\rVert_{3\mathfrak{p}}
&=\left\lVert\left\lVert W_s-W_t\right\rVert\right\rVert_{3\mathfrak{p}}
\leq \left\lVert \left\lVert W_{s}-W_{t}\right\rVert\right\rVert_{p}\leq \sqrt{s-t}\sqrt{d+2p}\leq \sqrt{a_d}\sqrt{s-t}\nonumber\\
&\leq (a_d)^{\frac{p}{2}\frac{\beta}{p}}\sqrt{s-t}\leq V_d(t,x)\sqrt{s-t}.\label{a08}
\end{align}
Next, \eqref{a04g}, the fact that $c\leq c^\beta$,
\eqref{c03}, and \eqref{c04} show for all 
$d\in \N$,
$x\in \R^d$, $t\in [0,T]$ that
\begin{align}\label{a04h}
\max\{\lvert g_d(x)\rvert ,\lvert Tf_d(t,x,0)\rvert\}\leq
c^\beta d^\beta(d^c+\lVert x\rVert)^\beta
\leq 
2^\beta c^\beta d^\beta(d^{2c}+\lVert x\rVert^2)^\frac{\beta}{2}
= (\varphi_d(x))^\frac{\beta}{p}\leq V_d(t,x)
.
\end{align}
This,  \eqref{a05a}, and \eqref{a09} show for all  $d\in \R^d$,
$t\in [0,T]$
 $x\in \R^d$, $w_1,w_2\in \R^{d+1}$ that
\begin{align}
&\lvert
f(t,x,w_1)-f(t,y,w_2)\rvert\nonumber\\&\leq \sum_{\nu=0}^{d}\left[
L_\nu^d\Lambda^d_\nu(T)
\lvert\pr^d_\nu(w_1-w_2) \rvert\right]
+\frac{1}{T}\frac{cd^\beta(d^c+\lVert x\rVert)^\beta
+cd^\beta(d^c+\lVert y\rVert)^\beta
}{2}\frac{\lVert x-y\rVert}{\sqrt{T}}\nonumber\\
&\leq 
\sum_{\nu=0}^{d}\left[
L_\nu^d\Lambda^d_\nu(T)
\lvert\pr^d_\nu(w_1-w_2) \rvert\right]
+\frac{1}{T}\frac{V(t,x)+V(t,y)}{2}\frac{\lVert x-y\rVert}{\sqrt{T}}
,\label{a05b}
\end{align}
\begin{align}
\lvert 
g(x)-g(y)\rvert&\leq 
\frac{cd^\beta(d^c+\lVert x\rVert)^\beta
+cd^\beta(d^c+\lVert y\rVert)^\beta
}{2}\frac{\lVert x-y\rVert}{\sqrt{T}}\leq 
\frac{V(T,x)+V(T,y)
}{2}\frac{\lVert x-y\rVert}{\sqrt{T}}.\label{a04}
\end{align}
Next, Jensen's inequality, the fact that
$3\exponentP\leq p$, 
a standard result on moments of standard normal distribution, and \eqref{a10} show for all $d\in \N$, $i\in [1,d]\cap \Z$, $s\in [0,T)$, $t\in (s,T]$ that
\begin{align}
\left\lVert \frac{W^i_s-W^i_t}{s-t}\right\rVert_{3\mathfrak{p}}
\leq 
\frac{1}{\sqrt{s-t}}
\left\lVert \frac{W^i_s-W^i_t}{\sqrt{s-t}}\right\rVert_{p}
\leq \frac{2^{\frac{p}{2}}\Gamma (\tfrac{p+1}{2})}{\sqrt{\pi}}\frac{1}{\sqrt{s-t}}
\leq \frac{c}{\sqrt{s-t}}.\label{a07}
\end{align}
Now, \cref{a20}
(applied for every $d\in \N$ with
$d\gets d$, $\Theta\gets \Theta$,
$T\gets T$, $\mathfrak{p}\gets \mathfrak{p}$, 
$\exponentV\gets 3\mathfrak{p}$,
$\exponentZ\gets 3\mathfrak{p}$,
$\exponentX\gets 3\mathfrak{p}$,
$c\gets c$, $(L_i)_{i\in [1,d]\cap\Z} \gets(L^d_i)_{i\in [1,d]\cap\Z}$,
$\lVert\cdot \rVert\gets \lVert\cdot \rVert$,
$\Lambda\gets \Lambda^d$, $\pr \gets \pr^d$, 
$f\gets f_d$, $g\gets g_d$, 
$V\gets V_d$,
$F\gets F_d$, $\varrho\gets \varrho$, $(\mathfrak{r}^\theta)_{\theta\in \Theta}\gets (\mathfrak{r}^\theta)_{\theta\in \Theta}$, $\mathcal{X}\gets \mathcal{X}^{d}$,
$\mathcal{Z}\gets \mathcal{Z}^{d}$,
$(U^{\theta}_{n,m})_{\theta\in \Theta,n,m\in \Z} \gets (U^{d,\theta}_{n,m})_{\theta\in \Theta,n,m\in \Z}$, $\exponentFirstNorm\gets3$
 in the notation of \cref{a20}), \eqref{a04h}, \eqref{a05b}, \eqref{a11}, \eqref{a08}, \eqref{a07}, \eqref{c04}, \eqref{a04}, \eqref{c02}, \eqref{a03}, 
the fact that $\forall \,d\in \N\colon \max \{ c,48e^{86c^6T^3}\}\leq V_d$ (cf.\ \eqref{c03} and \eqref{c04}),
and the assumptions on distributions and independence show that the following items hold.
\begin{enumerate}[(A)]
\item\label{a16k} For all $d,n,m\in \N$, $t\in [0,T)$, $x\in \R^d$ we have that $U^{d,\theta}_{n,m}(t,x)$ is measurable.
\item \label{a17k} For all $d\in \N$ there exists a unique measurable function $\mathfrak{u}_d\colon [0,T)\times \R^d\to \R^{d+1}$ such that for all $t\in [0,T)$, $x\in \R^d$
we have that
\begin{align}\label{b01k}
\max_{\nu\in [0,d]\cap\Z}\sup_{\tau\in [0,T), \xi\in \R^d}
\left[\Lambda^d_\nu(T-\tau)\frac{\lvert\pr^d_\nu(\mathfrak{u}_d(\tau,\xi))\rvert}{V_d(\tau ,\xi)}\right]<\infty,
\end{align}
\begin{align}\max_{\nu\in [0,d]\cap\Z}\left[
\E\!\left [\left\lvert g_d(\mathcal{X}^{d,0,t,x}_{T} )\pr^d_\nu(\mathcal{Z}^{d,0,t,x}_{T})\right\rvert \right] + \int_{t}^{T}
\E \!\left[\left\lvert
f_d(r,\mathcal{X}^{d,0,t,x}_{r},\mathfrak{u}_d(r,\mathcal{X}^{0,t,x}_{r}))\pr_\nu(\mathcal{Z}^{0,t,x}_{r})\right\rvert\right]dr\right]<\infty,
\end{align}
and
\begin{align}
\mathfrak{u}_d(t,x)=\E\!\left [g_d(\mathcal{X}^{d,0,t,x}_{T} )\mathcal{Z}^{d,0,t,x}_{T} \right] + \int_{t}^{T}
\E \!\left[
f_d(r,\mathcal{X}^{d,0,t,x}_{r},\mathfrak{u}_d(r,\mathcal{X}^{d,0,t,x}_{r}))\mathcal{Z}^{d,0,t,x}_{r}\right]dr.\label{a05r}
\end{align} 
\item  For all
$n,m\in \N$, $t\in [0,T)$, $\nu\in [0,d]\cap\Z$, $x\in \R^d$ we have that
\begin{align}\label{b40k}
\Lambda^d_\nu(T-t)\left\lVert\pr^d_{\nu}\!\left(U^{d,0}_{n,m}(t,x)-\mathfrak{u}_d(t,x)\right)\right\rVert_\exponentP\leq 
32^n(\exponentP-1)^{\frac{n}{2}}e^{nc^2T}
\exp\! \left(\frac{m^{1.5\exponentP}}{3\exponentP}\right)
m^{-n/2} V^3(t,x).
\end{align}
\end{enumerate}
First, \eqref{a16k} implies \eqref{a16l}. Next, \eqref{a17k}, e.g., \cite[Theorem~1]{Poh2024}
(applied for every $d\in \N$ with $\mu\gets 0 $, $\sigma\gets \mathrm{Id}_{\R^{d\times d}}$,
$(X_{t,s}^x)_{t\in [0,T], s\in [t,T],x\in \R^d} \gets 
(\mathcal{X}^{d,0,t,x}_s)_{t\in [0,T], s\in [t,T],x\in \R^d} $,
$(Z_{t,s}^x)_{t\in [0,T), s\in (t,T],x\in \R^d} \gets 
(\mathcal{Z}^{d,0,t,x}_s)_{t\in [0,T), s\in (t,T],x\in \R^d} $
in the notation of \cite[Theorem~1]{Poh2024}), and the assumption that for all $d\in \N$, $v_d$ is the classical solution of \eqref{c10} show for all $d\in \N $ that $u_d=\mathfrak{u}_d$ where $u_d:=(v_d,\nabla v_d)$. 
This and \eqref{b40k} show that
there exists $\kappa\in (0,\infty)$ such that
 for all 
$d\in \N$ it holds that
\begin{align}\label{b40t} \begin{split} 
&
\sup_{{\nu\in [0,d]\cap\Z}}
\sup_{x\in [-\mathbf{k},\mathbf{k}]^d}
\Lambda^d_\nu(T)\left\lVert\pr^d_{\nu}\!\left(U^{d,0}_{n,m}(0,x)-{u}_d(0,x)\right)\right\rVert_\exponentP\\
&\leq 
32^n(\exponentP-1)^{\frac{n}{2}}e^{nc^2T}
\exp\! \left(\frac{m^{1.5\exponentP}}{3\exponentP}\right)
m^{-n/2}\sup_{x\in [-\mathbf{k},\mathbf{k}]^d} V^3(0,x)\\
&\leq \left[
32(\mathfrak{p}-1)^\frac{1}{2}e^{c^2T}\exp \!\left(
\frac{m^{1.5\mathfrak{p}}}{n}
\right)m^{-\frac{1}{2}}\right]^n \kappa d^{\kappa}
.\end{split}
\end{align}
Next, for every $\delta,\varepsilon\in (0,1)$ let 
\begin{align}
N_\varepsilon= \inf\left\{
n\in \N\colon 
\left[
32(\mathfrak{p}-1)^\frac{1}{2}e^{c^2T}\exp \!\left(
\frac{(M_n)^{1.5\mathfrak{p}}}{n}
\right)(M_n)^{-\frac{1}{2}}\right]^n<\varepsilon
\right\}, \label{c19}
\end{align}
\begin{align}
C_\delta= \sup_{\varepsilon\in (0,1)}\left[
\varepsilon^{2+\delta}(3M_{N_\varepsilon})^{N_\varepsilon}\right].
\label{c20}
\end{align}
For every $d\in \N$, $\epsilon\in (0,1)$ let
\begin{align}
\varepsilon(d,\epsilon)=\frac{\epsilon}{\kappa d^\kappa},\quad  n(d,\epsilon)=N_{\varepsilon(d,\epsilon)}.\label{c17}
\end{align}
Next,
\cite[Lemma~4.5]{HJKP2021} and the definition of $(M_n)_{n\in \N}$ show
 that
$\liminf_{j\to\infty}M_j=\infty$, $\limsup_{j\to\infty} \frac{(M_j)^{1.5\mathfrak{p}}}{j}<\infty$, and $\sup_{k\in \N}\frac{M_{k+1}}{M_k}<\infty $.
Then 
\cite[Lemma~5.1]{AJKP2024}
(applied with
$L\gets 1$, $T\gets
32(\mathfrak{p}-1)^{\frac{1}{2}}e^{c^2T}-1$, 
$p\gets 3\mathfrak{p}$,
$(m_{k})_{k\in \N}\gets (M_{k})_{k\in \N}$
in the notation of \cite[Lemma~5.1]{AJKP2024}), \eqref{c19}, and
 \eqref{c20} 
show for all $\delta,\varepsilon\in (0,1)$ that $N_\varepsilon+C_\delta<\infty$.
Next, \eqref{b40t}, \eqref{c19}, and \eqref{c17} show for all $d\in \N$, $\epsilon\in (0,1)$ that
\begin{align}\begin{split} 
&
\sup_{{\nu\in [0,d]\cap\Z}}
\sup_{x\in [-\mathbf{k},\mathbf{k}]^d}
\Lambda^d_\nu(T)\left\lVert\pr^d_{\nu}\!\left(U^{d,0}_{n(d,\epsilon),M_{n(d,\epsilon)}}(0,x)-{u}_d(0,x)\right)\right\rVert_\exponentP\\
=
&
\sup_{{\nu\in [0,d]\cap\Z}}
\sup_{x\in [-\mathbf{k},\mathbf{k}]^d}
\Lambda^d_\nu(T)\left\lVert\pr^d_{\nu}\!\left(U^{d,0}_{N_{\varepsilon(d,\epsilon)},M_{N_{\varepsilon(d,\epsilon)}}}(0,x)-{u}_d(0,x)\right)\right\rVert_\exponentP\\
&\leq \left[
32(\mathfrak{p}-1)^\frac{1}{2}e^{c^2T}\exp \!\left(
\frac{(M_{N_{\varepsilon(d,\epsilon)}})^{1.5\mathfrak{p}}}{n}
\right)(M_{N_{\varepsilon(d,\epsilon)}})^{-\frac{1}{2}}\right]^n \kappa d^{\kappa}\\
&\leq \varepsilon(d,\varepsilon)\kappa d^\kappa=\epsilon.\end{split}\label{c16}
\end{align}
Next,
the family $(\mathrm{RV}_{d,n,m})_{d,m\in \N,n\in \N_0}$ satisfies the following recursive inequality: For all $d,n,m\in \N$ we have that
\begin{align}
\mathrm{RV}_{d,n,m}\leq dm^n+\sum_{\ell=0}^{n-1}
\left[m^{n-\ell } (d+1+ \mathrm{RV}_{d,\ell,m}+\1_\N(\ell )
\mathrm{RV}_{d,\ell-1,m}
)\right]
\end{align}
(cf. \cite[Display~(176)]{HJK2022}). Then 
\cite[Lemma~3.14]{beck2020overcomingElliptic} ensures for all $d,m,n\in \N$ that 
$\mathrm{RV}_{d,n,m}\leq 2d(3m)^n$. Hence, for all $d,N\in \N$ it holds that
\begin{align}
\sum_{n=1}^{N} \mathrm{RV}_{d,n,M_n}\leq 
2d\sum_{n=1}^{N}(3M_n)^n
\leq 
2d\sum_{n=1}^{N}(3M_N)^n
=\frac{2d(3M_N) ( 3(M_N)^N-1)}{3M_N-1}
\leq 4d(3M_N)^N.
\end{align}
Hence, \eqref{c17} and \eqref{c20} show for all $d\in \N$, $\epsilon\in (0,1)$ that
\begin{align}
&
\sum_{n=1}^{n(d,\epsilon)} \mathrm{RV}_{d,n,M_n}
=
\sum_{n=1}^{N_{\varepsilon(d,\epsilon)}} \mathrm{RV}_{d,n,M_n}
\leq 4d(3M_{N_{\varepsilon(d,\epsilon)}})^{N_{\varepsilon(d,\epsilon)}}\leq 4d C_\delta(\varepsilon(d,\epsilon))^{-(2+\delta)}
=4dC_\delta \left(\frac{\epsilon}{\kappa d^\kappa}\right)^{-(2+\delta)}\nonumber\\
&\leq 4dC_\delta (\kappa d^\kappa)^{2+\delta}\epsilon^{-(2+\delta)}.
\end{align}
This, 
\eqref{c16}, and the fact that $\forall\,d\in \N,\delta,\epsilon\in (0,1)\colon n(d,\epsilon)+C_\delta<\infty$ complete the proof of \cref{a41}.
\end{proof}

\section{Numerical experiments}\label{s12}
\begin{setting}\label{s39}
Let  $\Theta=\bigcup_{n\in \N}\Z^n$,
$T\in (0,\infty)$, $d\in \N$,
$f\in C(\R^{d+1},\R)$, $g\in C(\R^d,\R)$. 
Let $v\in C^{1,2}( [0,T]\times \R^d, \R)$ be an at most polynomially growing function. Assume for every $t\in (0,T)$, $x\in \R^d$ that
\begin{align}
\frac{\partial v}{\partial t}(t,x)+(\Delta_x v )(t,x)+
f(v(t,x), (\nabla_x v)(t,x))=0,\quad v(T,x)=g(x).\label{s09}
\end{align}
Let $\varrho\colon \{(\tau,\sigma)\in [0,T)^2\colon\tau<\sigma \}\to\R$ satisfy for all $t\in [0,T)$, $s\in (t,T)$ that
\begin{align}
\varrho(t,s)=\frac{1}{\mathrm{B}(\tfrac{1}{2},\tfrac{1}{2})}\frac{1}{\sqrt{(T-s)(s-t)}}.\label{p05}
\end{align}
Let $ (\Omega,\mathcal{F},\P)$ be a probability space.
Let $\mathfrak{r}^\theta\colon \Omega\to (0,1) $, $\theta\in \Theta$, be independent and identically distributed\ random variables and satisfy for all
$b\in (0,1)$
that 
\begin{align}
\P(\mathfrak{r}^0\leq b)=
\frac{1}{\mathrm{B}(\tfrac{1}{2},\tfrac{1}{2})}
\int_{0}^b\frac{dr}{\sqrt{r(1-r)}}.
\end{align}
Let $W^{\theta}\colon[0,T]\times \Omega\to \R^d$, $\theta\in \Theta$, be independent standard Brownian motions.
Assume that
$
(W^{\theta})_{\theta\in \Theta}
$ and
$(\mathfrak{r}^\theta)_{\theta\in \Theta}$  are independent. 
For every  $\theta\in \Theta$, $t\in [0,T]$, 
$s\in [t,T]$,
$x\in \R^d$ let \begin{align}
\mathcal{X}^{\theta,t,x}_s=x+W^{\theta}_s-W^{\theta}_t
.
\end{align}
For every $d\in\N$, $\theta\in \Theta$, $t\in [0,T)$, 
$s\in (t,T]$,
$x\in \R^d$ let 
\begin{align}
\mathcal{Z}^{\theta,t,x}_s=\left(1,\frac{W^{\theta}_s-W^{\theta}_t}{s-t}\right)
.
\end{align}
Let $
U^{\theta}_{n,m}\colon [0,T)\times \R^d\times\Omega\to \R^{d+1}$,
 $n,m\in \Z$, $\theta\in \Theta$, satisfy for all $n,m\in \N$, 
$\theta\in \Theta$,
$t\in [0,T)$, $x\in\R^d$ that
$U_{-1,m}^{\theta}(t,x)=U^{\theta}_{0,m}(t,x)=0$ and
\begin{align} \begin{split} 
&
U^{\theta}_{n,m}(t,x)= (g(x),0)+\sum_{i=1}^{m^n}
\frac{g(\mathcal{X}^{(\theta,0,-i),t,x}_T )-g_d(x) }{m^n}\mathcal{Z}^{(\theta,0,-i),t,x}_{T}\\
& +\sum_{\ell=0}^{n-1}\sum_{i=1}^{m^{n-\ell}} \frac{\left(f\circ U^{(\theta,\ell,i)}_{\ell,m}-
\1_\N(\ell)
f\circ U^{(\theta,\ell,-i)}_{\ell-1,m})\right)\!\left(t+(T-t) \mathfrak{r}^{(\theta,\ell,i)},
\mathcal{X}^{(\theta,\ell,i),t,x}_{t+(T-t) \mathfrak{r}^{(\theta,\ell,i)}}\right)
\mathcal{Z}^{(\theta,\ell,i),t,x}_{t+(T-t) \mathfrak{r}^{(\theta,\ell,i)}}}{m^{n-\ell}\varrho(t,t+(T-t)\mathfrak{r}^{(\theta,\ell,i)})}.\end{split}
\end{align}
Let $M\colon \N\to \N$ satisfy for all $n\in\N $ that
$
M_n=\max \{k\in \N\colon k\leq\exp (\lvert\ln(n)\rvert^{1/2})\}.
$
\end{setting}
In \cref{p01} below we implement the MLP approximations, see lines \ref{p02}--\ref{p03}.
Note that the global variable \texttt{count} is introduced to count
the number of real-valued random variables needed for the MLP
approximations. The functions $f$ and $g$ will be defined in each example. 
The function \texttt{bmi}, defined in lines \ref{p05c}--\ref{p05a}, generates a Brownian increment needed for the MLP approximation.
The code was written in
Julia (see \url{https://julialang.org}). We used a laptop with 16GB RAM, 12th Gen Intel Core i5-1240P
x 16, Operating System: Ubuntu 22.04.4 LTS 64 bit.

\begin{listing}\label{p01}
The following code should be saved under the name \texttt{MLP.jl}.
\begin{lstlisting}
function varrho(t,s)
    return 1.0/ sqrt( (T-s)*(s-t) ) / beta(0.5,0.5)
end
function bmi(d,s,t)|\label{p05c}|
    global count; count=count+d
    rand(Normal(),d)*sqrt(t-s)
end|\label{p05a}|
function U(n,m,t,x)|\label{p02}|
    d=length(x)
    if (n==0)
        return zeros(d+1);
    end    
    s=[g(x);zeros(d)]    
    for i in 1:m^n
        D=bmi(d,t,T);
        s=s+ (g(x+D)-g(x))*[1;D/(T-t)]/(m^n);               
    end
    for l in 0:(n-1)
        for i in 1:(m^(n-l))
            global count
            count=count+1
            r=rand(Beta(0.5,0.5));
            D=bmi(d,t,t+(T-t)*r);
            if (l>0)
                s=s+ ( f(U(l,m,t+(T-t)*r,x+D)) 
					- f(U(l-1,m,t+(T-t)*r,x+D)) )
					*[1;D/(T-t)/r]/(m^(n-l))/varrho(t,t+(T-t)*r)
            else
                s=s+ f(U(l,m,t+(T-t)*r,x+D))
				*[1;D/(T-t)/r]/(m^(n-l))/varrho(t,t+(T-t)*r)
            end            
        end
    end
    return s;
end|\label{p03}|
function M(n)    
    return  floor(exp(sqrt(log(n))))
end
\end{lstlisting}
\end{listing}

\begin{example}\label{s10}
Assume \cref{s39}. Assume that $T=1$, $d=100$, and assume
for all $w=(w_{i})_{i\in [0,d]\cap\Z}\in \R^{d+1}$, $x=(x_i)_{i\in [1,d]}\in \R^d$ that  
$f(w)=\sin (\frac{1}{d}\sum_{i=0}^{d}w_i)$ and $g(x)=\sin (\frac{1}{d}\sum_{i=1}^{d}x_i)$.
In this example the PDE is
\begin{align}
\frac{\partial v}{\partial t}(t,x)+(\Delta_x v )(t,x)+
\sin\! \left(\frac{1}{d}\left(v(t,x)+ \sum_{i=1}^{d}\frac{\partial v}{\partial x_i}(t,x) \right)\right)=0,\quad v(T,x)=\sin\!\left(\frac{1}{d}\sum_{i=1}^{d}x_i\right).\label{s09a}
\end{align}
for $t\in [0,T]$, $x\in \R^d$.
In \cref{s11} below we provide a Julia code to calculate 
$U^0_{n,n}(0,0)$ for $n\in [1,7]\cap\Z$. Here we use the MLP approximations with $n=m=6$ as reference solutions.

The output is demonstrated in Table~\ref{t01} and Figure~\ref{t02}. On the left-hand side of Figure~\ref{t02} we present the
relation between the computational error and the variable \texttt{count} and on the right-hand side of Figure~\ref{t02} we
present the relation between the computational error and the runtime in seconds. The polylines in Figure~\ref{t02}
represent the difference between the corresponding MLP approximation introduced in \cref{s39} and the
corresponding reference solution in $7$ cases 
$n\in [1,7]\cap\Z$.
\end{example}
\begin{listing}[Code for \cref{s10}]
\label{s11}
The following code should be saved in the same folder with \texttt{MLP.jl}
under the name \texttt{example-1.jl}. To run the code we type \texttt{julia example-1.jl}. The outputs will be contained in \texttt{example-1.csv}, \texttt{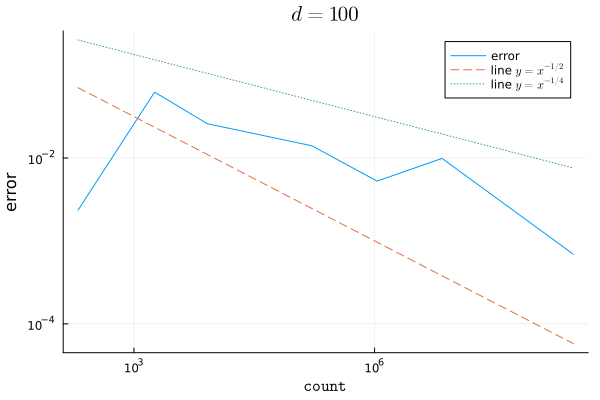}, and \texttt{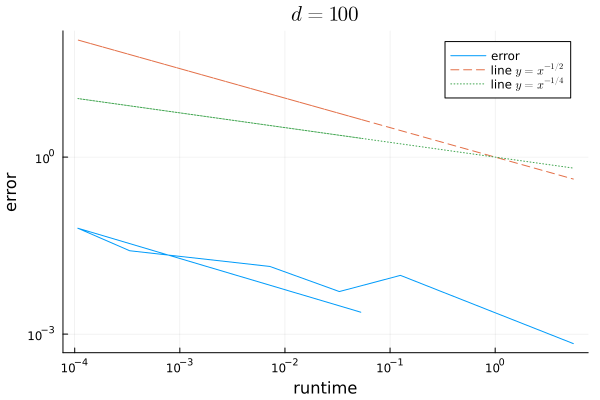}. 
Note that before
running the code we may need to first install the packages used here, see lines \ref{s11j}--\ref{s12j}.
\begin{lstlisting}
using Distributions, LinearAlgebra, SpecialFunctions,|\label{s11j}|
Plots,LaTeXStrings,DataFrames,CSV|\label{s12j}|
include("MLP.jl")
function f(u)    
    return  sin(sum(u)/d);
end
function g(x)
    return sin(sum(x)/d)
end
T=1.0; t=0.0; x=zeros(100); d=length(x);
count=0; Nmax=7; diff=zeros(Nmax); 
RV=zeros(Nmax); RT=zeros(Nmax)
u2=U(6,6,t,x)[1]
for i in 1:Nmax
    global count; count=0; ti=time()
    u1=U(i,M(i),t,x)[1]
    RV[i]=count; RT[i]=time()-ti; 
    diff[i]=abs(u1-u2)    
end
df=DataFrame( error=diff,RT=RT,RV=RV)
CSV.write("example-1.csv", df)
plot(RV,[diff RV.^(-1/2) RV.^(-1/4)], xaxis=:log10, yaxis=:log10, 
    label=["error" "line "*L"y=x^{-1/2}" "line "*L"y=x^{-1/4}"], 
    ls=[:solid :dash :dot], xlabel="RV", ylabel="error",
    title=L"d=%$d")
savefig("example-1-a.png") 
plot(RT,[diff RT.^(-1/2) RT.^(-1/4)], xaxis=:log10, yaxis=:log10, 
    label=["error" "line "*L"y=x^{-1/2}"  "line "*L"y=x^{-1/4}"], 
    ls=[:solid :dash :dot], xlabel="runtime", ylabel="error",
    title=L"d=%$d")
savefig("example-1-b.png") 
\end{lstlisting}
\end{listing}

\begin{table}
\begin{tabular}{|l|l|l|l|}
\hline
Case&error	&Runtime in seconds&	\texttt{count}\\
\hline
$n=1$&0.00235206956774445	&0.0526549816131592	&201\\
$n=2$&0.062343788083941	&0.000107049942016602	&1810\\
$n=3$&0.0258988111429571	&0.000333070755004883	&8246\\
$n=4$&0.0140285938864811	&0.00718116760253906	&165894\\
$n=5$&0.00526720016182484	&0.032905101776123	&1072581\\
$n=6$&0.0099027526534033	&0.125412940979004	&6933471\\
$n=7$&0.000690565128244561	&5.53902101516724	&300556996\\
\hline
\end{tabular}
\caption{Numerical experiments for \cref{s10}}
\label{t01}
\end{table}
\begin{figure}
\includegraphics[width=0.49\textwidth]{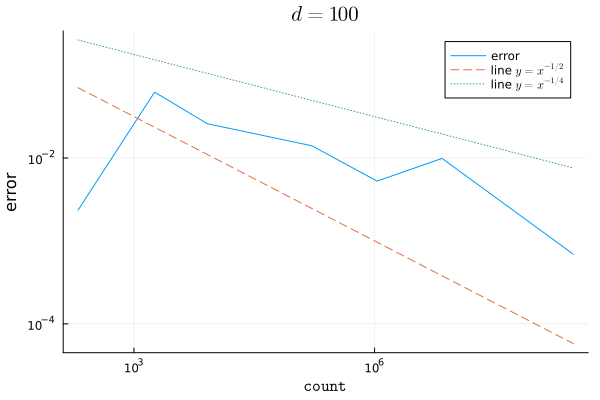}
\includegraphics[width=0.49\textwidth]{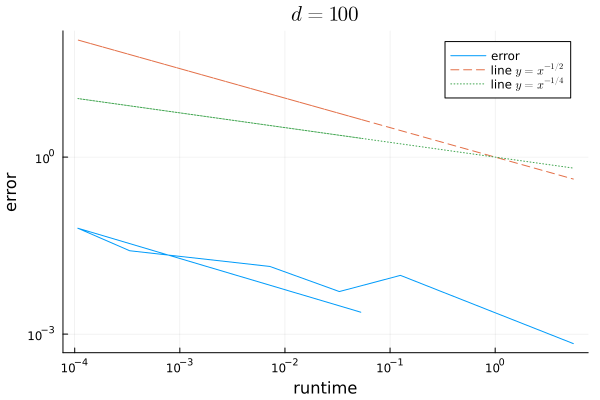}
\caption{Numerical experiments for \cref{s10}. The left-hand side shows the error in dependence on \texttt{count}. The right-hand side shows the error in dependence on the runtime in seconds.}
\label{t02}
\end{figure}

\begin{example}\label{s10c}
Assume \cref{s39}. Assume that $T=1$, $d=100$, and assume
for all $w=(w_{i})_{i\in [0,d]\cap\Z}\in \R^{d+1}$, $x=(x_i)_{i\in [1,d]}\in \R^d$ that  
$f(w)=\cos  (w_1)$ and $g(x)=\sin \bigl(\frac{1}{d}\sum_{i=1}^{d}\lvert x_i\rvert^2\bigr)$.
In this example the PDE \eqref{s09} is
\begin{align}
\frac{\partial v}{\partial t}(t,x)+(\Delta_x v )(t,x)+
\cos\! \left(\frac{\partial v}{\partial x_1}(t,x)\right)=0,\quad v(T,x)=\sin\!\left(\frac{1}{d}\sum_{i=1}^{d}\lvert x_i\rvert^2\right).\label{s09d}
\end{align}
for $t\in [0,T]$, $x\in \R^d$.
In \cref{s19} below we provide a Julia code to calculate 
$U^0_{n,n}(0,0)$ for $n\in [1,7]\cap\Z$. Here we use the MLP approximations with $n=m=6$ as reference solutions.

The output is demonstrated in Table~\ref{t03} and Figure~\ref{f06}. On the left-hand side of Figure~\ref{f06} we present the
relation between the computational error and the variable \texttt{count} and on the right-hand side of Figure~\ref{f06} we
present the relation between the computational error and the runtime in seconds. The polylines in Figure~\ref{f06}
represent the difference between the corresponding MLP approximation introduced in \cref{s39} and the
corresponding reference solution in $7$ cases 
$n\in [1,7]\cap\Z$.
\end{example}
\begin{listing}
[Code for \cref{s10c}]\label{s19}
The following code should be saved in the same folder with \texttt{MLP.jl}
under the name \texttt{example-2.jl}. To run the code we type \texttt{julia example-2.jl}. The outputs will be contained in \texttt{example-2.csv}, \texttt{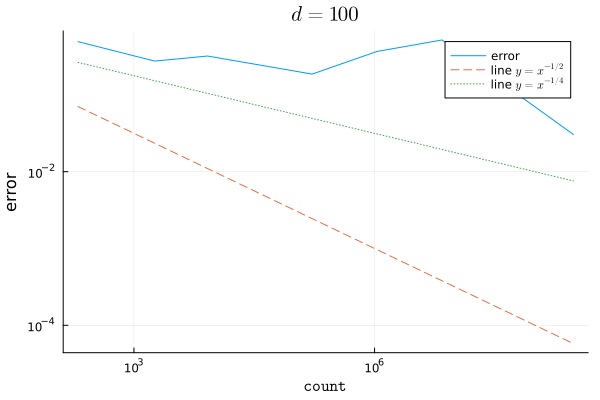}, and \texttt{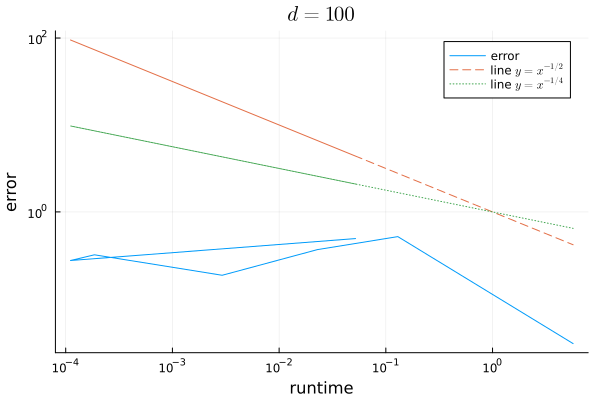}. 
Note that before
running the code we may need to first install the packages used here, see lines \ref{s11b}--\ref{s12b}.
\begin{lstlisting}
|\label{s11b}|using Distributions, LinearAlgebra, SpecialFunctions,
|\label{s12b}|Plots,LaTeXStrings,DataFrames,CSV
include("MLP.jl")
function f(u)    
    return  cos(u[2]);
end
function g(x)
    return sin(sum(x.^2)/d)
end
T=1.0; t=0.0; x=zeros(100); d=length(x);
count=0; Nmax=7; diff=zeros(Nmax); 
RV=zeros(Nmax); RT=zeros(Nmax)
u2=U(6,6,t,x)[1]
for i in 1:Nmax
    global count; count=0; ti=time()
    u1=U(i,M(i),t,x)[1]
    RV[i]=count; RT[i]=time()-ti; 
    diff[i]=abs(u1-u2)    
end
df=DataFrame( error=diff,RT=RT,RV=RV)
CSV.write("example-2.csv", df)
plot(RV,[diff RV.^(-1/2) RV.^(-1/4)], xaxis=:log10, yaxis=:log10, 
    label=["error" "line "*L"y=x^{-1/2}" "line "*L"y=x^{-1/4}"], 
    ls=[:solid :dash :dot], xlabel=L"\texttt{count}", ylabel="error",
    title=L"d=%$d")
savefig("example-2-a.png") 

plot(RT,[diff RT.^(-1/2) RT.^(-1/4)], xaxis=:log10, yaxis=:log10, 
    label=["error" "line "*L"y=x^{-1/2}"  "line "*L"y=x^{-1/4}"], 
    ls=[:solid :dash :dot], xlabel="runtime", ylabel="error",
    title=L"d=%$d")
savefig("example-2-b.png") 
\end{lstlisting}
\end{listing}

\begin{figure}
\includegraphics[width=0.49\textwidth]{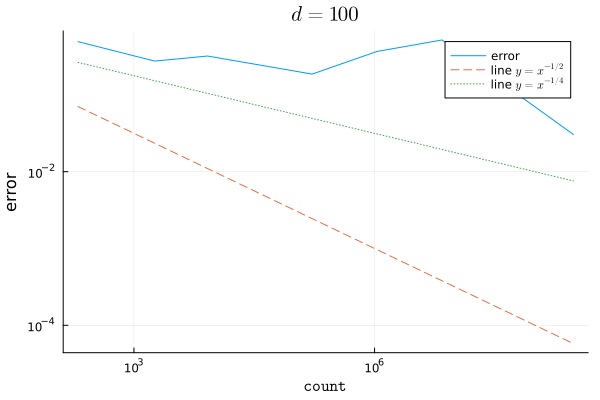}
\includegraphics[width=0.49\textwidth]{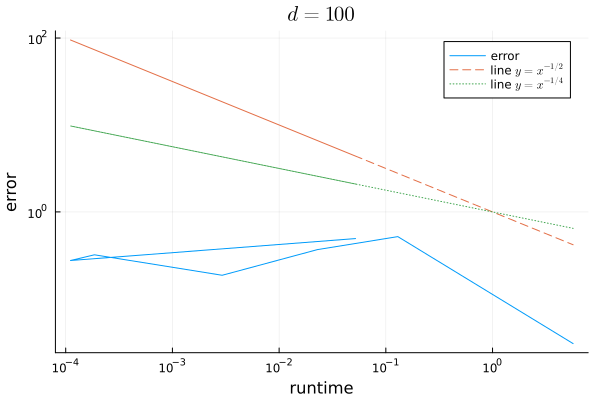}
\caption{Numerical experiments for \cref{s10c}. The left-hand side shows the error in dependence on \texttt{count}. The right-hand side shows the error in dependence on the runtime in seconds.}
\label{f06}
\end{figure}
\begin{table}
\begin{tabular}{|l|l|l|l|}
\hline
Case&error	&Runtime in seconds&	\texttt{count}\\
\hline
$n=1$&0.493905019512897	&0.0519950389862061	&201\\
$n=2$&0.27624793330089	&0.000110864639282227	&1810\\
$n=3$&0.321358205654427	&0.000185966491699219	&8246\\
$n=4$&0.186869537059728	&0.00292420387268066	&165894\\
$n=5$&0.368851915081614	&0.0228722095489502	&1072581\\
$n=6$&0.520064762414605	&0.12950611114502	&6933471\\
$n=7$&0.0305862555122101	&5.68948888778687	&300556996\\
\hline
\end{tabular}
\label{t03}
\caption{Numerical experiments for \cref{s10c}}
\end{table}

\bibliographystyle{acm}
\bibliography{References}

\newcommand{\dummy}[1]{}
\begin{thebibliography}{10}

\bibitem{AJKP2024}
{\sc Ackermann, J., Jentzen, A., Kuckuck, B., and Padgett, J.~L.}
\newblock {Deep neural networks with ReLU, leaky ReLU, and softplus activation
  provably overcome the curse of dimensionality for space-time solutions of
  semilinear partial differential equations}.
\newblock {\em arXiv:2406.10876\/} (2024).

\bibitem{al2022extensions}
{\sc Al-Aradi, A., Correia, A., Jardim, G., de~Freitas~Naiff, D., and Saporito,
  Y.}
\newblock Extensions of the deep {G}alerkin method.
\newblock {\em Applied Mathematics and Computation 430\/} (2022), 127287.

\bibitem{beck2020deep}
{\sc Beck, C., Becker, S., Cheridito, P., Jentzen, A., and Neufeld, A.}
\newblock Deep learning based numerical approximation algorithms for stochastic
  partial differential equations and high-dimensional nonlinear filtering
  problems.
\newblock {\em arXiv preprint arXiv:2012.01194\/} (2020).

\bibitem{beck2021deep}
{\sc Beck, C., Becker, S., Cheridito, P., Jentzen, A., and Neufeld, A.}
\newblock Deep splitting method for parabolic {PDE}s.
\newblock {\em SIAM Journal on Scientific Computing 43}, 5 (2021),
  A3135--A3154.

\bibitem{beck2019machine}
{\sc Beck, C., E, W., and Jentzen, A.}
\newblock Machine learning approximation algorithms for high-dimensional fully
  nonlinear partial differential equations and second-order backward stochastic
  differential equations.
\newblock {\em Journal of Nonlinear Science 29\/} (2019), 1563--1619.

\bibitem{beck2020overcomingElliptic}
{\sc Beck, C., Gonon, L., and Jentzen, A.}
\newblock Overcoming the curse of dimensionality in the numerical approximation
  of high-dimensional semilinear elliptic partial differential equations.
\newblock {\em arXiv preprint arXiv:2003.00596\/} (2020).

\bibitem{beck2020overcoming}
{\sc Beck, C., Hornung, F., Hutzenthaler, M., Jentzen, A., and Kruse, T.}
\newblock Overcoming the curse of dimensionality in the numerical approximation
  of {A}llen--{C}ahn partial differential equations via truncated full-history
  recursive multilevel {P}icard approximations.
\newblock {\em Journal of Numerical Mathematics 28}, 4 (2020), 197--222.

\bibitem{beck2020overview}
{\sc Beck, C., Hutzenthaler, M., Jentzen, A., and Kuckuck, B.}
\newblock An overview on deep learning-based approximation methods for partial
  differential equations.
\newblock {\em arXiv preprint arXiv:2012.12348\/} (2020).

\bibitem{becker2020numerical}
{\sc Becker, S., Braunwarth, R., Hutzenthaler, M., Jentzen, A., and von
  Wurstemberger, P.}
\newblock Numerical simulations for full history recursive multilevel {P}icard
  approximations for systems of high-dimensional partial differential
  equations.
\newblock {\em arXiv preprint arXiv:2005.10206\/} (2020).

\bibitem{berner2020numerically}
{\sc Berner, J., Dablander, M., and Grohs, P.}
\newblock Numerically solving parametric families of high-dimensional
  {K}olmogorov partial differential equations via deep learning.
\newblock {\em Advances in Neural Information Processing Systems 33\/} (2020),
  16615--16627.

\bibitem{castro2022deep}
{\sc Castro, J.}
\newblock Deep learning schemes for parabolic nonlocal integro-differential
  equations.
\newblock {\em Partial Differential Equations and Applications 3}, 6 (2022),
  77.

\bibitem{CHW2022}
{\sc Cioica-Licht, P.~A., Hutzenthaler, M., and Werner, P.~T.}
\newblock Deep neural networks overcome the curse of dimensionality in the
  numerical approximation of semilinear partial differential equations.
\newblock {\em arXiv:2205.14398v1\/} (2022).

\bibitem{CHJ2021}
{\sc Cox, S., Hutzenthaler, M., and Jentzen, A.}
\newblock {Local Lipschitz continuity in the initial value and strong
  completeness for nonlinear stochastic differential equations}.
\newblock {\em arXiv:1309.5595\/} (2021).

\bibitem{ew2017deep}
{\sc E, W., Han, J., and Jentzen, A.}
\newblock Deep learning-based numerical methods for high-dimensional parabolic
  partial differential equations and backward stochastic differential
  equations.
\newblock {\em Communications in Mathematics and Statistics 5}, 4 (2017),
  349--380.

\bibitem{weinan2021algorithms}
{\sc E, W., Han, J., and Jentzen, A.}
\newblock Algorithms for solving high dimensional {PDE}s: from nonlinear
  {M}onte {C}arlo to machine learning.
\newblock {\em Nonlinearity 35}, 1 (2021), 278.

\bibitem{ew2018deep}
{\sc E, W., and Yu, B.}
\newblock The deep {R}itz method: A deep learning-based numerical algorithm for
  solving variational problems.
\newblock {\em Commun Math Stat 6}, 1 (2018), 1--12.

\bibitem{frey2022convergence}
{\sc Frey, R., and K{\"o}ck, V.}
\newblock Convergence analysis of the deep splitting scheme: the case of
  partial integro-differential equations and the associated {FBSDE}s with
  jumps.
\newblock {\em arXiv preprint arXiv:2206.01597\/} (2022).

\bibitem{frey2022deep}
{\sc Frey, R., and K{\"o}ck, V.}
\newblock Deep neural network algorithms for parabolic {PIDE}s and applications
  in insurance mathematics.
\newblock In {\em Methods and Applications in Fluorescence}. Springer, 2022,
  pp.~272--277.

\bibitem{GPW2022}
{\sc Germain, M., Pham, H., and Warin, X.}
\newblock Approximation error analysis of some deep backward schemes for
  nonlinear {PDE}s.
\newblock {\em SIAM Journal on Scientific Computing 44}, 1 (2022), A28--A56.

\bibitem{giles2019generalised}
{\sc Giles, M.~B., Jentzen, A., and Welti, T.}
\newblock Generalised multilevel picard approximations.
\newblock {\em arXiv preprint arXiv:1911.03188\/} (2019).

\bibitem{gnoatto2022deep}
{\sc Gnoatto, A., Patacca, M., and Picarelli, A.}
\newblock A deep solver for {BSDE}s with jumps.
\newblock {\em arXiv preprint arXiv:2211.04349\/} (2022).

\bibitem{gonon2023random}
{\sc Gonon, L.}
\newblock Random feature neural networks learn {B}lack-{S}choles type {PDE}s
  without curse of dimensionality.
\newblock {\em Journal of Machine Learning Research 24}, 189 (2023), 1--51.

\bibitem{gonon2021deep}
{\sc Gonon, L., and Schwab, C.}
\newblock Deep {R}e{LU} network expression rates for option prices in
  high-dimensional, exponential {L}{\'e}vy models.
\newblock {\em Finance and Stochastics 25}, 4 (2021), 615--657.

\bibitem{gonon2023deep}
{\sc Gonon, L., and Schwab, C.}
\newblock Deep {R}e{LU} neural networks overcome the curse of dimensionality
  for partial integrodifferential equations.
\newblock {\em Analysis and Applications 21}, 01 (2023), 1--47.

\bibitem{grohs2023proof}
{\sc Grohs, P., Hornung, F., Jentzen, A., and von Wurstemberger, P.}
\newblock A proof that artificial neural networks overcome the curse of
  dimensionality in the numerical approximation of {B}lack--{S}choles partial
  differential equations.
\newblock {\em Mem. Am. Math. Soc. 284\/} (2023).

\bibitem{han2018solving}
{\sc Han, J., Jentzen, A., and E, W.}
\newblock Solving high-dimensional partial differential equations using deep
  learning.
\newblock {\em Proceedings of the National Academy of Sciences 115}, 34 (2018),
  8505--8510.

\bibitem{han2020convergence}
{\sc Han, J., and Long, J.}
\newblock Convergence of the deep {BSDE} method for coupled {FBSDE}s.
\newblock {\em Probability, Uncertainty and Quantitative Risk 5\/} (2020),
  1--33.

\bibitem{han2019solving}
{\sc Han, J., Zhang, L., and E, W.}
\newblock Solving many-electron {S}chr{\"o}dinger equation using deep neural
  networks.
\newblock {\em Journal of Computational Physics 399\/} (2019), 108929.

\bibitem{hure2020deep}
{\sc Hur{\'e}, C., Pham, H., and Warin, X.}
\newblock {Deep backward schemes for high-dimensional nonlinear PDEs}.
\newblock {\em Mathematics of Computation 89}, 324 (2020), 1547--1579.

\bibitem{hutzenthaler2019multilevel}
{\sc Hutzenthaler, M., Jentzen, A., and Kruse, T.}
\newblock {On multilevel Picard numerical approximations for high-dimensional
  nonlinear parabolic partial differential equations and high-dimensional
  nonlinear backward stochastic differential equations}.
\newblock {\em Journal of Scientific Computing 79}, 3 (2019), 1534--1571.

\bibitem{hutzenthaler2021multilevel}
{\sc Hutzenthaler, M., Jentzen, A., and Kruse, T.}
\newblock {Multilevel Picard iterations for solving smooth semilinear parabolic
  heat equations}.
\newblock {\em Partial Differential Equations and Applications 2}, 6 (2021),
  1--31.

\bibitem{HJK2022}
{\sc Hutzenthaler, M., Jentzen, A., and Kruse, T.}
\newblock Overcoming the curse of dimensionality in the numerical approximation
  of parabolic partial differential equations with gradient-dependent
  nonlinearities.
\newblock {\em Foundations of Computational Mathematics 22\/} (2022), 905--966.

\bibitem{HJKN2020}
{\sc Hutzenthaler, M., Jentzen, A., Kruse, T., and Nguyen, T.~A.}
\newblock {Multilevel Picard approximations for high-dimensional semilinear
  second-order PDEs with Lipschitz nonlinearities}.
\newblock {\em arXiv:2009.02484\/} (2020).

\bibitem{hutzenthaler2020proof}
{\sc Hutzenthaler, M., Jentzen, A., Kruse, T., and Nguyen, T.~A.}
\newblock A proof that rectified deep neural networks overcome the curse of
  dimensionality in the numerical approximation of semilinear heat equations.
\newblock {\em SN partial differential equations and applications 1\/} (2020),
  1--34.

\bibitem{HJK+2020}
{\sc Hutzenthaler, M., Jentzen, A., Kruse, T., Nguyen, T.~A., and von
  Wurstemberger, P.}
\newblock Overcoming the curse of dimensionality in the numerical approximation
  of semilinear parabolic partial differential equations.
\newblock {\em Proceedings of the Royal Society A: Mathematical, Physical and
  Engineering Sciences 476}, 2244 (2020), 20190630.

\bibitem{HJKP2021}
{\sc Hutzenthaler, M., Jentzen, A., Kuckuck, B., and Padgett, J.~L.}
\newblock {Strong $L^p$-error analysis of nonlinear Monte Carlo approximations
  for high-dimensional semilinear partial differential equations}.
\newblock {\em arXiv:2110.08297\/} (2021).

\bibitem{hutzenthaler2020overcoming}
{\sc Hutzenthaler, M., Jentzen, A., and von Wurstemberger, P.}
\newblock Overcoming the curse of dimensionality in the approximative pricing
  of financial derivatives with default risks.
\newblock {\em Electron. J. Probab. 25}, 101 (2020), 1--73.

\bibitem{HK2020}
{\sc Hutzenthaler, M., and Kruse, T.}
\newblock Multilevel {P}icard approximations of high-dimensional semilinear
  parabolic differential equations with gradient-dependent nonlinearities.
\newblock {\em SIAM Journal on Numerical Analysis 58}, 2 (2020), 929--961.

\bibitem{hutzenthaler2022multilevel}
{\sc Hutzenthaler, M., Kruse, T., and Nguyen, T.~A.}
\newblock Multilevel {P}icard approximations for {M}c{K}ean-{V}lasov stochastic
  differential equations.
\newblock {\em Journal of Mathematical Analysis and Applications 507}, 1
  (2022), 125761.

\bibitem{HN2022b}
{\sc Hutzenthaler, M., and Nguyen, T.~A.}
\newblock {Multilevel Picard approximations for high-dimensional decoupled
  forward-backward stochastic differential equations}.
\newblock {\em arXiv:2204.08511v\/} (2022).

\bibitem{HN2022a}
{\sc Hutzenthaler, M., and Nguyen, T.~A.}
\newblock {Multilevel Picard approximations of high-dimensional semilinear
  partial differential equations with locally monotone coefficient functions}.
\newblock {\em Applied Numerical Mathematics 181\/} (2022), 151--175.

\bibitem{ito2021neural}
{\sc Ito, K., Reisinger, C., and Zhang, Y.}
\newblock A neural network-based policy iteration algorithm with global
  {$H^2$}-superlinear convergence for stochastic games on domains.
\newblock {\em Foundations of Computational Mathematics 21}, 2 (2021),
  331--374.

\bibitem{jacquier2023deep}
{\sc Jacquier, A., and Oumgari, M.}
\newblock Deep curve-dependent {PDE}s for affine rough volatility.
\newblock {\em SIAM Journal on Financial Mathematics 14}, 2 (2023), 353--382.

\bibitem{jacquier2023random}
{\sc Jacquier, A., and Zuric, Z.}
\newblock Random neural networks for rough volatility.
\newblock {\em arXiv preprint arXiv:2305.01035\/} (2023).

\bibitem{JSW2021}
{\sc Jentzen, A., Salimova, D., and Welti, T.}
\newblock A proof that deep artificial neural networks overcome the curse of
  dimensionality in the numerical approximation of {K}olmogorov partial
  differential equations with constant diffusion and nonlinear drift
  coefficients.
\newblock {\em Communications in Mathematical Sciences 19}, 5 (2021),
  1167--1205.

\bibitem{lu2021deepxde}
{\sc Lu, L., Meng, X., Mao, Z., and Karniadakis, G.~E.}
\newblock Deep{XDE}: A deep learning library for solving differential
  equations.
\newblock {\em SIAM review 63}, 1 (2021), 208--228.

\bibitem{NNW2023}
{\sc Neufeld, A., Nguyen, T.~A., and Wu, S.}
\newblock {Multilevel Picard approximations overcome the curse of
  dimensionality in the numerical approximation of general semilinear PDEs with
  gradient-dependent nonlinearities}.
\newblock {\em arXiv:2311.11579\/} (2023).

\bibitem{neufeld2023multilevel}
{\sc Neufeld, A., Nguyen, T.~A., and Wu, S.}
\newblock Multilevel {P}icard approximations overcome the curse of
  dimensionality in the numerical approximation of general semilinear {PDE}s
  with gradient-dependent nonlinearities.
\newblock {\em arXiv preprint arXiv:2311.11579\/} (2023).

\bibitem{NSW2024}
{\sc Neufeld, A., Schmocker, P., and Wu, S.}
\newblock {Full error analysis of the random deep splitting method for
  nonlinear parabolic PDEs and PIDEs with infinite activity}.
\newblock {\em arXiv:2405.05192\/} (2024).

\bibitem{NW2022}
{\sc Neufeld, A., and Wu, S.}
\newblock {Multilevel Picard approximation algorithm for semilinear partial
  integro-differential equations and its complexity analysis}.
\newblock {\em arXiv:2205.09639v3\/} (2022).

\bibitem{NW2023}
{\sc Neufeld, A., and Wu, S.}
\newblock {Multilevel Picard algorithm for general semilinear parabolic PDEs
  with gradient-dependent nonlinearities}.
\newblock {\em arXiv:2310.12545\/} (2023).

\bibitem{nguwi2022deep}
{\sc Nguwi, J.~Y., Penent, G., and Privault, N.}
\newblock A deep branching solver for fully nonlinear partial differential
  equations.
\newblock {\em arXiv preprint arXiv:2203.03234\/} (2022).

\bibitem{nguwi2022numerical}
{\sc Nguwi, J.~Y., Penent, G., and Privault, N.}
\newblock Numerical solution of the incompressible {N}avier-{S}tokes equation
  by a deep branching algorithm.
\newblock {\em arXiv preprint arXiv:2212.13010\/} (2022).

\bibitem{nguwi2023deep}
{\sc Nguwi, J.~Y., and Privault, N.}
\newblock A deep learning approach to the probabilistic numerical solution of
  path-dependent partial differential equations.
\newblock {\em Partial Differential Equations and Applications 4}, 4 (2023),
  37.

\bibitem{Poh2024}
{\sc Pohl, K.}
\newblock {\em {Numerical approximation methods for semilinear partial
  differential equations with gradient-dependent nonlinearities}}.
\newblock PhD thesis, University of Duisburg--Essen, 2024.

\bibitem{raissi2019physics}
{\sc Raissi, M., Perdikaris, P., and Karniadakis, G.~E.}
\newblock Physics-informed neural networks: A deep learning framework for
  solving forward and inverse problems involving nonlinear partial differential
  equations.
\newblock {\em Journal of Computational physics 378\/} (2019), 686--707.

\bibitem{reisinger2020rectified}
{\sc Reisinger, C., and Zhang, Y.}
\newblock Rectified deep neural networks overcome the curse of dimensionality
  for nonsmooth value functions in zero-sum games of nonlinear stiff systems.
\newblock {\em Analysis and Applications 18}, 06 (2020), 951--999.

\bibitem{Rio2009}
{\sc Rio, E.}
\newblock Moment inequalities for sums of dependent random variables under
  projective conditions.
\newblock {\em Journal of Theoretical Probability 22\/} (2009), 146--163.

\bibitem{sirignano2018dgm}
{\sc Sirignano, J., and Spiliopoulos, K.}
\newblock {DGM}: A deep learning algorithm for solving partial differential
  equations.
\newblock {\em Journal of computational physics 375\/} (2018), 1339--1364.

\bibitem{zhang2020learning}
{\sc Zhang, D., Guo, L., and Karniadakis, G.~E.}
\newblock Learning in modal space: Solving time-dependent stochastic {PDEs}
  using physics-informed neural networks.
\newblock {\em SIAM Journal on Scientific Computing 42}, 2 (2020), A639--A665.

\end{thebibliography}

\end{document}